\documentclass[12pt]{extarticle}
\usepackage[utf8]{inputenc}
\usepackage[T1]{fontenc}

\usepackage{caption}
\oddsidemargin \evensidemargin
\usepackage{amssymb}
\usepackage{pmboxdraw}
\usepackage{amsmath}
\usepackage{bbm}
\usepackage{amsthm}
\usepackage{amssymb}
\usepackage[dvipsnames]{xcolor}
\usepackage{graphicx}
\usepackage{alphabeta}
\usepackage{tikz,pgfplots}
\usepackage{mathrsfs}
\usepackage{hyperref}
\usepackage[all]{xy}
\usepackage[greek,english]{babel}
\usepackage[export]{adjustbox}
\usepackage{hyperref}
\usepackage{bibspacing}
\hypersetup{
    colorlinks=true,
    linkcolor=orange!80!black,
    filecolor=magenta,      
    urlcolor=blue,
    citecolor=blue,
}
\usepackage{enumitem}

\definecolor{myblue}{rgb}{0.00000,0.44700,0.74100}
\definecolor{myorange}{rgb}{0.8500, 0.3250, 0.0980}
\definecolor{myyellow}{rgb}{0.9290, 0.6940, 0.1250}
\definecolor{mypurple}{rgb}{0.4940, 0.1840, 0.5560}
\definecolor{mygreen}{rgb}{0.4660, 0.6740, 0.1880}

\usepackage[frozencache,cachedir=minted-cache]{minted}
\setminted[julia]{frame=lines,
rulecolor=\color{white!80!black},
fontsize=\small,
numbers=right,
numbersep=-5pt,
obeytabs=true,
tabsize=4}

\newtheorem{theorem}{Theorem}[section]
\newtheorem{theorem*}[theorem]{Theorem*}

\newtheorem{lemma}[theorem]{Lemma}
\newtheorem{corollary}[theorem]{Corollary}
\newtheorem{proposition}[theorem]{Proposition}

\newtheorem{conjecture}[theorem]{Conjecture}
\theoremstyle{definition}
\newtheorem{examplex}[theorem]{Example}
\newenvironment{example}
{\pushQED{\qed}\examplex}
{\popQED\endexamplex}
\newtheorem{definition}[theorem]{Definition}
\newtheorem{remark}[theorem]{Remark}

\theoremstyle{remark}

\title{\bf Santal\'o Geometry of Convex Polytopes}
\author{Dmitrii Pavlov and Simon Telen}
\date{}

\begin{document}

\maketitle

\begin{abstract}
\noindent The Santal\'o point of a convex polytope is the interior point which leads to a polar dual of minimal volume. This minimization problem is relevant in interior point methods for convex optimization, where the logarithm of the dual volume is known as the universal barrier function. When translating the facet hyperplanes, the Santal\'o point traces out a semi-algebraic set. We describe and compute this geometry using algebraic and numerical techniques. We exploit connections with statistics, optimization and~physics. 
\end{abstract}

\small \noindent {\bf Keywords}: adjoints, Wachspress models, Santal\'o points, dual volume, positive geometries, convex optimization, linear programming, interior point methods\\
{\bf MSC}: 52B20, 
52A40, 
62R01, 
65H14 

\normalsize

\section{Introduction}

This article studies the (semi-)algebraic geometry of minimizing volumes of dual polytopes. Motivations include optimization, statistics and particle physics. To make this more precise, we start with some terminology. A polytope $P \subset \mathbb{R}^m$ is the convex hull of finitely many points. If $P$ has dimension $m$, then each point $y$ in its interior defines a dual polytope 
\[ (P-y)^\circ \, = \, \{ z \in (\mathbb{R}^m)^\vee \, : \, \langle y'-y, z \rangle \leq 1, \, \, \text{for all} \, \, y' \in P \}. \]
The function $y \mapsto {\rm vol}_m \, (P-y)^\circ$ is strictly convex on the interior of $P$. In fact, this is true when $P$ is replaced by any convex body, see the proof of Proposition 1(i) in \cite{meyer1998santalo}. It follows that there is a unique minimizer $y^* \in {\rm int}(P)$. That point is called the \emph{Santal\'o point} of $P$:
\begin{equation} \label{eq:ystar}
y^* \, = \, \underset{y \in {\rm int}(P)}{\operatorname{argmin}} \, \,  {\rm vol}_m \, (P-y)^\circ \, = \, \underset{y \in {\rm int}(P)}{\operatorname{argmin}}  \, \, \int_{(P-y)^\circ} {\rm d} z_1 \cdots  {\rm d} z_m.  \end{equation}
A special property of polytopes, compared to general convex bodies, is that our volume function is \emph{rational}. It follows from Theorems 3.1 and 3.2 in \cite{gaetz2020positive} that 
\begin{equation} \label{eq:volrational}
{\rm vol}_m \, (P-y)^\circ \, = \, \gamma \cdot \frac{\alpha_P(y)}{\ell_1(y) \cdot \cdots \cdot \ell_k(y)},
\end{equation}
where $\gamma$ is a nonzero real constant, $\ell_i(y) = 0$ is an affine-linear equation defining the $i$-th facet hyperplane of $P$, and $\alpha_P(y)$ is the adjoint polynomial. We will recall a formula for $\alpha_P$ in Section \ref{sec:2}. 
Having established the identity \eqref{eq:volrational}, computing the Santal\'o point of $P$ comes down to minimizing a convex rational function or, equivalently, its logarithm. 
\begin{example}[$m = 2, k = 5$] \label{ex:pentagon1}
    We consider the pentagon $P$ in $\mathbb{R}^2$ given by the inequalities 
    \[ y_1 + \frac{1}{5} \geq 0, \quad y_2 + \frac{1}{5} \geq 0, \quad 2 y_1 + 2 y_2 + \frac{1}{5} \geq 0, \quad -2y_1 -y_2  + \frac{1}{5} \geq 0, \quad -y_1 -2y_2 + \frac{1}{5} \geq 0 .\]
    It is shown, together with the poles and zeros of ${\rm vol}_2 (P-y)^\circ$, in Figure \ref{fig:pentagon1} (left). We have
    \begin{equation} \label{eq:volpentagon}
    {\rm vol}_2 (P-y)^\circ \, = \, \frac{1}{125} \dfrac{-50y_1^2 - 25y_1y_2 + 15y_1 - 50y_2^2 + 15y_2 + 11}{(y_1+ \frac{1}{5})(y_2+ \frac{1}{5})(2y_1+2 y_2 + \frac{1}{5})(-2y_1 -y_2 + \frac{1}{5})(-y_1 -2y_2 + \frac{1}{5})}.    \end{equation}
    The Santal\'o point minimizes this function on ${\rm int}(P)$: ${y^* = (-0{.}00311069, 
    -0{.}00311069)}$.
      \begin{figure}
        \centering
        \includegraphics[height = 4.5cm]{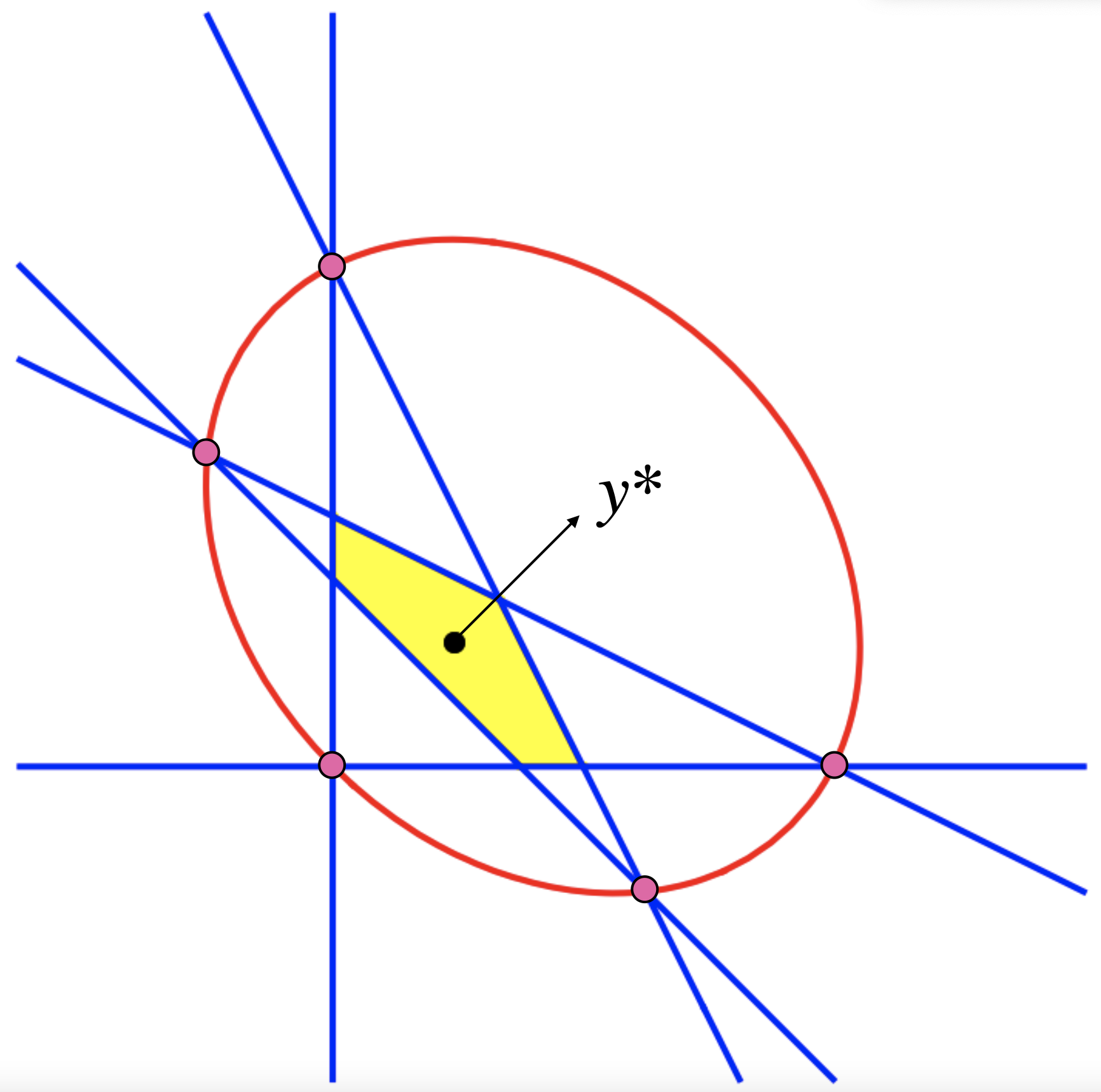}
        \quad \quad \quad \quad 
        \includegraphics[height = 4.5cm]{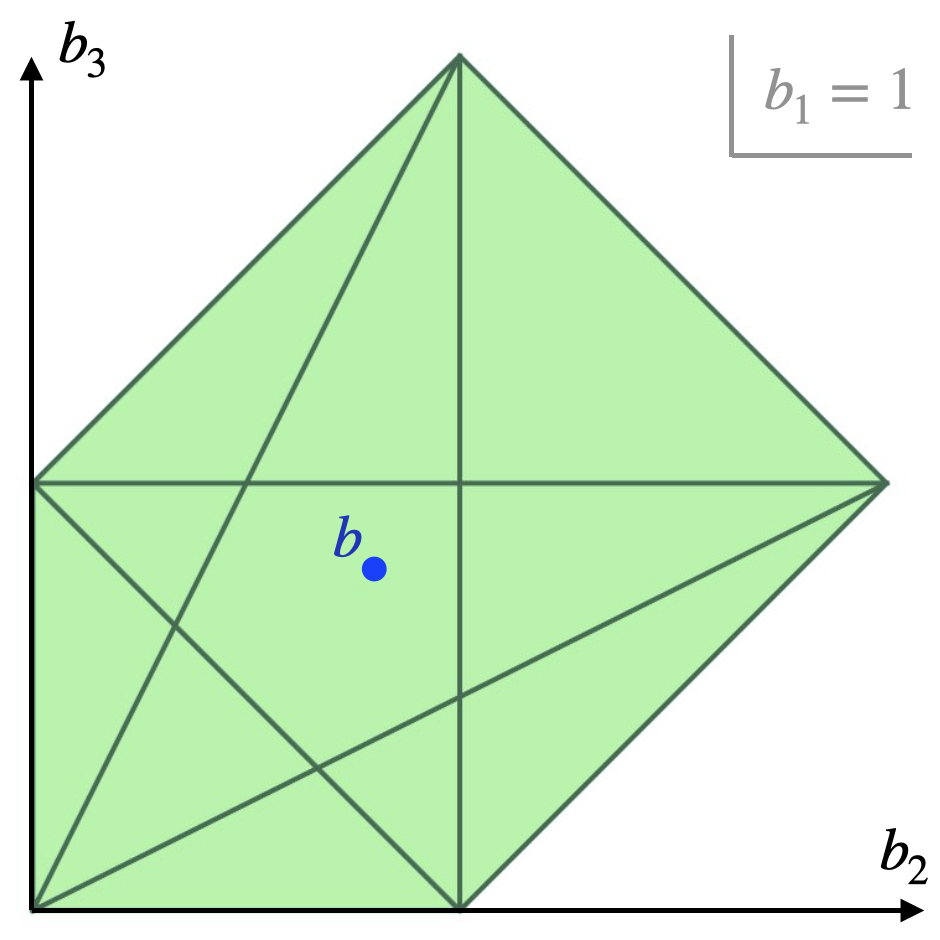}
        \caption{Left: the pentagon $P$ from Example \ref{ex:pentagon1}, together with its adjoint curve (red) and facet hyperplanes (blue). Right: a two-dimensional slice of the chamber complex ${\cal C}_A$.}
        \label{fig:pentagon1}
    \end{figure}
\end{example}

The first motivation for computing Santal\'o points comes from convex optimization \cite{nesterov1994interior}. In that context, $P$ is the feasible region of a linear program, whose optimal solution is typically a vertex of $P$. Interior point methods approximate that vertex by first optimizing a strictly convex (barrier) function. The resulting interior optimizer is then tracked to the optimal vertex by varying a regularization parameter. For more details, see \cite{guler1996barrier,nesterov1994interior}, where \eqref{eq:volrational} is called the universal barrier function. For a summary, see the introduction of \cite{sturmfels2024toric}. 

We are interested in how the Santal\'o point varies when the facet hyperplanes of $P$ are translated. 
More precisely, we fix a nonnegative $(d \times n)$-matrix $A \in \mathbb{R}^{d\times n}_{\geq 0}$ of rank $d$, none of whose columns is the zero vector, and consider the fibers of the projection $A: \mathbb{R}^{n}_{\geq 0}\rightarrow \mathbb{R}^d$: 
\[ P_b \, = \, \{ x \in \mathbb{R}^n_{\geq 0} \,: \, Ax = b \}, \quad b \in {\rm pos}(A). \]
Here ${\rm pos}(A)$ is the image of $A: \mathbb{R}^{n}_{>0}\rightarrow \mathbb{R}^d$. If $b$ lies in ${\rm pos}(A)$, then $P_b$ is a polytope of dimension $m = n-d$. A point $x$ in its relative interior defines a full-dimensional polytope $P_b - x$ in the $(n-d)$-dimensional vector space $\ker A \simeq \mathbb{R}^{n-d}$. We define
\begin{equation} \label{eq:Vx} V \, : \, \mathbb{R}^n_{>  0} \, \longrightarrow \, \mathbb{R}_{\geq 0},\quad x \, \longmapsto \, {\rm vol}_{n-d} \, (P_{Ax} - x)^\circ. \end{equation}
This is defined up to a scaling factor, which depends on the choice of basis for $\ker A$. We prove that this global volume function is piecewise rational, meaning that it is a rational function when restricted to certain $n$-dimensional subcones of $\mathbb{R}^{n}_{> 0}$ (Proposition \ref{prop:rational}). These subcones correspond to the cells of the chamber complex $\mathcal{C}_A$ associated to $A$, see for instance \cite{billera1993duality}. Moreover, on each of these subcones, $V$ is homogeneous of degree $d-n$ (Proposition~\ref{prop:rational}). 

Each fiber $P_b$ has a unique Santal\'o point. This defines a natural section of $A: \mathbb{R}^{n}_{>0}\rightarrow \mathbb{R}^d$:
\begin{equation} \label{eq:argmin} x^*(b) \, =\, \underset{x \in {\rm int}(P_b)}{\operatorname{argmin}} \, \, V(x) \,. \end{equation}
The map $x^*: {\rm pos}(A) \rightarrow \mathbb{R}^n_{>0}$ is piecewise algebraic. Its image is called the Santal\'o patchwork. We show that the Santal\'o patchwork is a union of $d$-dimensional basic semi-algebraic sets, one for each $d$-dimensional cell in the chamber complex ${\cal C}_A$ (Corollary \ref{cor:homeo}). We give inequalities for each of its pieces (called Santal\'o patches), and bound the degree of their Zariski closures.  
\begin{example}[$d = 3, n = 5$]\label{ex:AB}
    The pentagon in Example \ref{ex:pentagon1} is the fiber $P_b - x$ for the data 
\begin{equation} \label{eq:Abpentagon}
A \, = \, \begin{pmatrix}
    1 & 1 &  1 & 1 & 1   \\ 
    2 & 1 & 0 & 1 & 0   \\ 
    1 & 2 & 0 & 0 & 1  
\end{pmatrix}, \quad \quad b \, = \, \frac{1}{5}\begin{pmatrix}
    5 \\ 4 \\ 4
\end{pmatrix}, \quad 
x = \frac{1}{5}\begin{pmatrix}
    1 & 1 & 1 & 1 & 1
\end{pmatrix}^T .
\end{equation}
The coordinates $y_1$ and $y_2$ in Example \ref{ex:pentagon1} are with respect to the following basis of $\ker A$: 
\[
B \, = \, \frac{1}{18}\begin{pmatrix}
    5 & -4 & 2 & -6 & 3 \\ 
    -4 & 5 & 2 & 3 & -6
\end{pmatrix}^T.  \]
The columns of $A$ are the vertices of a pentagon in $\mathbb{R}^3$. They define the polyhedral complex shown in Figure \ref{fig:pentagon1} (right). The Chamber complex ${\cal C}_A$ is the polyhedral fan over that complex. There are 11 three-dimensional cells. Our $b$ lies in the central pentagonal cell. For any $x \in \mathbb{R}^5_{\geq 0}$ such that $Ax$ lies in this cell, we have the following formula for the function $V(x)$: 
\begin{equation} \label{eq:Vxpentagon} 
V(x) \, = \, \frac{3x_1x_2x_3 + 2x_1x_3x_5 + 2x_1x_4x_5 + 2x_2x_3x_4 + 2x_2x_4x_5}{x_1x_2x_3x_4x_5}.
\end{equation}
To match this with \eqref{eq:volpentagon}, use $Ax = b$ and $B^T x = y$ to switch from $x$- to $(b,y)$-coordinates and substitute $b = (1,4/5,4/5)$. A different rational function is needed when $b$ belongs to a different cell, because the combinatorial type of $P_b$ changes. For instance, one checks that for $b = (1,6/5,4/5)$, $P_b$ is a quadrilateral. Each cell in ${\cal C}_A$ gives a patch of the Santal\'o patchwork, which is a 3-dimensional semi-algebraic set in $\mathbb{R}^5_{\geq 0}$. Intersecting this with the 4-dimensional simplex $\{ \sum_{i=1}^5 x_i = 1 \}$ and projecting to $\mathbb{R}^3$, we obtain Figure \ref{fig:SPpentagon}.
\begin{figure}
    \centering
    \includegraphics[height = 4.5cm]{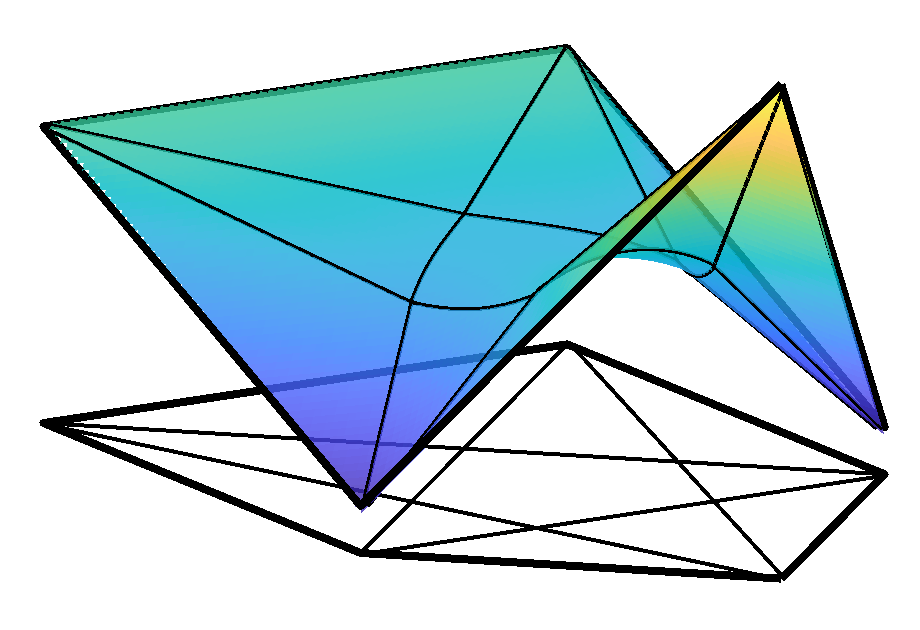} \quad \quad 
    \includegraphics[height = 4.5cm]{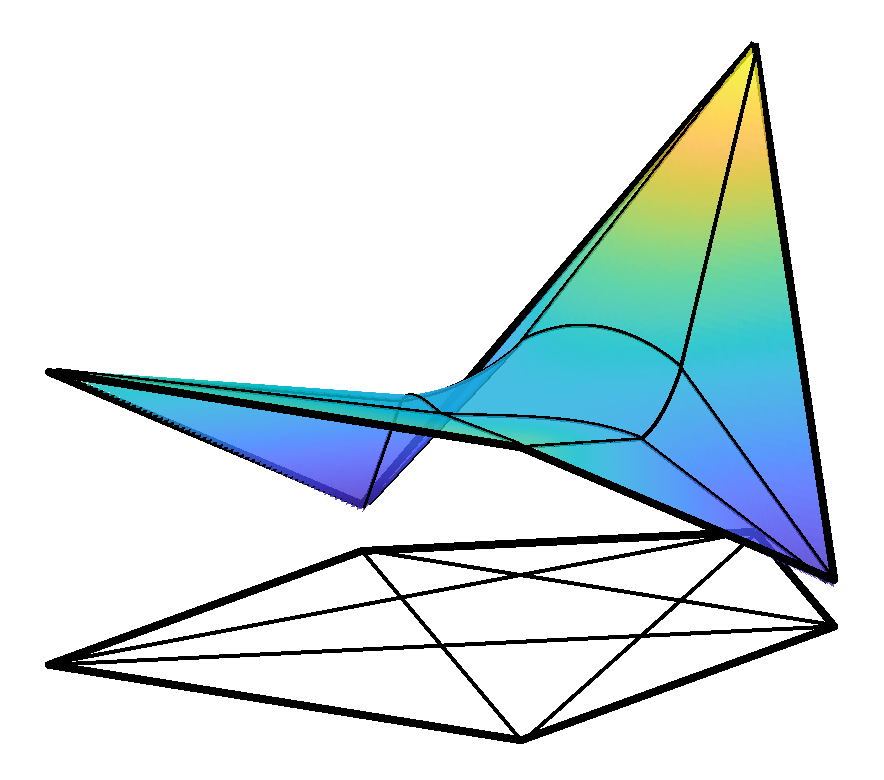}
    \caption{Visualization of the Santal\'o patchwork for $A$ from \eqref{eq:Abpentagon}.}
    \label{fig:SPpentagon}
\end{figure}
\end{example}

The pentagon from Examples \ref{ex:pentagon1} and \ref{ex:AB} will serve as a running example throughout. 

Understanding the degrees of Santal\'o patches relies on insights from algebraic statistics \cite{drton2008lectures}. Minimizing the logarithm of the dual volume has the interpretation of maximum likelihood estimation for a particular class of statistical models, called Wachspress models \cite[Section~2]{kohn2020moment}. Every righthand side vector $b \in {\rm pos}(A)$ defines a Wachspress model. The maximum likelihood degree (ML degree) \cite{catanese2006maximum} of this model is constant for generic $b$ in the interior of a cell in the chamber complex. We conjecture that, under mild genericity assumptions, it gives a lower bound for the degree of the corresponding Santal\'o patch, see Conjecture \ref{conj:irred} and Proposition \ref{prop:Iub}. Example \ref{ex:degreeVSMLdeg} gives evidence for the claim that this lower bound is close to the actual degree of the Santal\'o patch. We show how to compute the ML degree numerically, and Proposition \ref{prop:MLdeg} gives a formula for polygons, assuming Conjecture \ref{conj:genericadjoint}. 

Our outline is as follows. Section \ref{sec:2} studies the volume function \eqref{eq:Vx}. Sections \ref{sec:3} and \ref{sec:4} describe the Santal\'o patchwork and its Zariski closure. Section \ref{sec:5} makes the link to Wachspress models.
In Section \ref{sec:6} we discuss homotopy based methods for computing Santal\'o points. First, we use monodromy to compute the Santal\'o point of some fiber $P_{b_0}$. Next, we compute the Santal\'o point of a new fiber $P_{b_1}$ from that of $P_{b_0}$, such that $b_0$ and $b_1$ belong to the same chamber of ${\cal C}_A$. We use numerical homotopy continuation \cite{sommese2005numerical} to track $x^*(b_0)$ to $x^*(b_1)$ along a smooth path on the Santal\'o patchwork. We also show how to solve a linear program by using the universal barrier function and computing the Santal\'o point. Our algorithms are implemented in a Julia package \texttt{Santalo.jl}, which is available at 
\url{https://mathrepo.mis.mpg.de/Santalo} \cite{mathrepo}. Finally, Section \ref{sec:7} contains a summary of ideas for future research. 

Our work fits nicely into a broader story of semi-algebraic sets in convex optimization, algebraic statistics and particle physics. Different strictly convex objective functions used in interior point methods give rise to other interesting geometric objects, see \cite{de2012central,sturmfels2024toric}. For the log-barrier function $V(x) = -\sum_{i=1}^n \log x_i$, the role of the Santal\'o patchwork is played by the positive reciprocal linear space associated to the row span of the matrix $A$. The Santal\'o point is replaced by the analytic center. Entropic regularization uses $V(x) = \sum_{i=1}^n  x_i\log x_i - x_i$ and leads naturally to consider the positive toric $d$-fold associated to $A$, with the Birch point being its unique intersection with $P_b$. From a statistical point of view, these scenarios correspond to maximum likelihood estimation for linear models and exponential families respectively. Next to optimization and statistics, the dual volume function \eqref{eq:volrational} shows up in particle physics as the canonical function of $P$, viewed as a positive geometry \cite{arkani2017positive}. This enters in the proof of Propostion \ref{prop:rational}. For some specific polytopes, $V(x)$ is a scattering amplitude \cite{arkani2018scattering}. Recently, dual volumes have been used in the study of toric singularities \cite{moraga2021bounding}.

All of these connections motivate our effort to study the Santal\'o geometry of polytopes. Our work provides new theoretical insights into Santal\'o points, and practical tools for computing them. It leads to several new possible research directions, as summarized in Section~\ref{sec:7}.

\section{Dual volumes of polytopes} \label{sec:2}

To avoid confusion, below we write $Q \subset \mathbb{R}^m$ for a full-dimensional polytope (where, usually, $m = n-d$), and $P_b \subset \mathbb{R}^n$ for the $(n-d)$-dimensional fibers of $A: \mathbb{R}^n_{\geq 0} \rightarrow {\rm pos}(A)$. 

This section describes the dual volume function \eqref{eq:volrational} of a full-dimensional polytope $Q \subset \mathbb{R}^m$.
We start with the numerator of this rational function, called the \emph{adjoint polynomial} $\alpha_Q(y)$. We say that $Q$ is \emph{simple} if each vertex is adjacent to exactly $m$ facets.

Suppose $Q$ is simple and has minimal facet representation 
\begin{equation}\label{eq:minfacets}
   Q = \{y'\in \mathbb{R}^m \,:\, \langle w_i, y'\rangle + c_i \geq 0, \, i=1,\ldots,k\}. 
\end{equation}
Here $w_i \in \mathbb{R}^m$ and $c_i \in \mathbb{R}$. The adjoint polynomial of $Q$, introduced by Warren \cite{warren1996barycentric}, is 
\begin{equation}\label{eq:adjointvol}
    \alpha_Q(y) \, = \,  \mathrm{vol}_m(Q-y)^\circ \cdot \prod\limits_{i=1}^k (\langle w_i, y\rangle + c_i) . 
\end{equation}
For completeness, we include a proof of a convenient formula for $\alpha_Q(y)$. We collect the vectors~$w_i$ in an $m \times k$ matrix $W$ and write $W_I$ for the submatrix of columns indexed by $I \subset \{1,\ldots, k\}$. Let ${\cal V}(Q)$ be the set of vertices of $Q$.
For each~$v \in \mathcal{V}(Q)$, we let~$I(v) = \{i\,:\,\langle w_i, v \rangle +c_i = 0\} \subset \{1,\ldots, k\}$ be the $m$-element index set of the facets containing $v$. 

\begin{proposition} \label{prop:adjformula}
    For a simple full-dimensional polytope $Q \subset \mathbb{R}^m$ with minimal facet representation \eqref{eq:minfacets} the adjoint polynomial $\alpha_Q(y)$ is given by 
    \begin{equation}\label{eq:adjoint}
        \alpha_Q(y) \, = \, \sum\limits_{v \in \mathcal{V}(Q)} |\det W_{I(v)}| \cdot \prod\limits_{i \not\in I(v)}\left(\langle w_i,y\rangle + c_i \right).
    \end{equation}
\end{proposition}

\begin{proof}
    For $y\in \mathrm{int}(Q)$ the translated polytope $Q-y$ has the following facet representation:
    $$Q-y\, =\, \{y'\in \mathbb{R}^m \,:\,\langle w_i, y'\rangle + (\langle w_i, y\rangle + c_i)\geq 0, \, i=1,\ldots,k\}.$$
    The dual polytope is then simplicial and can be described as
    $$(Q-y)^\circ \, = \,  \mathrm{conv}\left(\left\{\dfrac{w_i}{\langle w_i, y\rangle + c_i},\, i=1,\ldots,k \right\}\right).$$
    We compute its volume as the sum of volumes over pieces of its triangulation:
    \begin{multline*}
        \mathrm{vol}_m(Q-y)^\circ =  \sum\limits_{v\in \mathcal{V}(Q)} \mathrm{vol}_m\left(\mathrm{conv}\left(\{0\} \cup \bigcup\limits_{i \in I(v)} \left\{\dfrac{w_i}{\langle w_i, y\rangle + c_i}\right\}\right)\right) =\\ \sum\limits_{v \in \mathcal{V}(Q)} \left|\det W_{I(v)}\right| \prod_{i\in I(v)}\left(\langle w_i, y \rangle + c_i \right)^{-1}.
    \end{multline*}
    Since by definition $\alpha_Q(y) =\mathrm{vol}_m(Q-y)^\circ \cdot \prod\limits_{i=1}^k (\langle w_i, y\rangle + c_i)$, we get the formula in \eqref{eq:adjoint}.
\end{proof}

To avoid confusion, we point out that what we call the adjoint of $Q$ is the adjoint of the dual polytope $Q^\circ$ in some of the literature \cite{kohn2020projective,warren1996barycentric}. The variety inside $\mathbb{R}^m$ defined by $\alpha_Q$ is the \emph{adjoint hypersurface} associated to $Q$, see \cite{kohn2020projective}. 
When the facet hyperplanes of $Q$ form a simple arrangement (that is, the intersection of any $i$ hyperplanes has codimension $i$), the adjoint hypersurface is the unique hypersurface of minimal degree interpolating the \emph{residual arrangement} of $Q$. This arrangement is the union of all affine spaces that are contained in the intersections of facet hyperplanes but do not contain a face of $Q$ \cite[Theorem 6]{kohn2020projective}. In Figure \ref{fig:pentagon1} (left), the residual arrangement consists of 5 points defining a unique adjoint conic.

We now switch to the setting outlined in the Introduction, where $m = n-d$ and the polytope $Q$ arises as a fiber $P_b$ of the linear projection $A : \mathbb{R}^n_{\geq 0} \rightarrow \mathbb{R}^d$ for some $A \in \mathbb{R}^{d \times n}_{\geq 0}$. If $x$ is an interior point of $P_b$, then the translate $P_b-x$ is a full-dimensional polytope inside $\ker A \cong \mathbb{R}^{n-d}$. We are interested in minimizing its dual volume $\mathrm{vol}_{n-d}(P_b-x)^\circ$ with respect to $x$. In order to treat this problem algebraically, we will first project $P_b$ to $\ker A$. To do so, fix an $(n \times (n-d))$-matrix $B$ whose columns span $\ker A$. The projection of $P_b$ is denoted by $Q_b = B^T \cdot P_b$ and the coordinates $y$ on $\ker A$ are induced from $y = B^T x$.

By construction, the matrix obtained by concatenating $A$ and $B^T$ vertically is an $n\times n$ matrix of full rank. It therefore defines an invertible coordinate change 
\begin{equation} \label{eq:coordchange}
 \begin{pmatrix}
    b \\
    y
\end{pmatrix}
\,=\,
\begin{pmatrix}
    A\\
    B^T
\end{pmatrix}
x.
\end{equation}
This means that in order to compute the Santal\'o point $x^*(b)$ of $P_b$, it is sufficient to compute the Santal\'o point $y^*(b)$ of $Q_b$ and then apply the inverse coordinate change:
\begin{equation} \label{eq:bytox}
    x^*(b) = \begin{pmatrix}
    A \\
    B^T
\end{pmatrix}^{-1}
\begin{pmatrix}
    b\\
    y^*(b)
\end{pmatrix}.
\end{equation}

We will now study the dual volume function ${\rm vol}_{n-d} \, (Q_b-y)^\circ$ for the polytope $Q_b$. Our aim is to show that this is a piecewise rational function of $y$ and $b$. A key role will be played by the \emph{chamber complex} $\mathcal{C}_A$ of $\mathrm{cone}(A)=\overline{\mathrm{pos}(A)}$, the conical hull of the columns of~$A$. 

Let $a_i$ denote the $i$-th column of $A$. For a nonempty subset $\sigma \subset [n]=\{1, \ldots ,n\}$ we define $A_\sigma = \{a_i: i\in \sigma\}$ to be the submatrix with columns indexed by $\sigma$. 

\begin{definition} \label{def:chambercomplex}
    For $b \in \mathrm{cone}(A)$, define the \emph{chamber} $C_b:=\bigcap_{{\rm cone}(A_\sigma) \ni b} \mathrm{cone}(A_\sigma)$. The \emph{chamber complex} of $A$~is the collection of all such chambers: 
    $$\mathcal{C}_A:=\{C_b:b\in \mathrm{cone}(A)\}.$$
    In the rest of this article, full-dimensional chambers are called \emph{cells} of~$\mathcal{C}_A$.
\end{definition}
For more details on the chamber complex and its properties, see \cite{billera1993duality} and \cite[Chapter 5]{de2010triangulations}.

\begin{proposition}  \label{prop:simple}
    For each $b$ in the interior of a cell $C \in \mathcal{C}_A$, the $(n-d)$-dimensional polytopes $P_b$ and $Q_b$ are simple, and so are their facet hyperplane arrangements. As $b$ varies over ${\rm int}(C)$, the combinatorial types of $P_b$ and $Q_b$ are equal and constant.
\end{proposition}

\begin{proof}
    Since $b$ is in $\mathrm{pos}(A)$, the interior of $\mathrm{cone}(A)$, $P_b$ has dimension $n-d$. Since every vertex $v$ of $P_b$ is a solution of $Av = b$ with $v_i = 0$ for $n-d$ entries of $v$ \cite[Theorem 2.4]{bertsimas1997introduction}, it is on exactly $n-d$ facet hyperplanes, and the polytope $P_b$ is simple. For essentially the same reason, the facet hyperplane arrangement of $P_b$ is simple for any $b\in {\rm int}(C)$.
    
    The affine span of $P_b$ is parallel to $\ker A$. The matrix $B$ whose columns span $\ker A$ defines a projection to $\ker A$, and the projected polytope $Q_b = B^T \cdot P_b$ has the same dimension and combinatorial type as $P_b$. The fact that the combinatorial type of $P_b$ stays the same as $b$ varies over a given chamber $C \in \mathcal{C}_A$ appears as Theorem 18 in \cite{alexandr2023maximum}.
\end{proof}

\begin{example}
The columns of the matrix $A$ from Example \ref{ex:AB} define the vertices of a pentagon shown in Figure \ref{fig:pentagon1} (right).
The positive hull $\mathrm{pos}(A)$ is a cone over this pentagon, and the chamber complex $\mathcal{C}_A$ is the fan over the polyhedral complex obtained by taking the common refinement of all triangulations of this pentagon.
The chamber complex has $11$ cells: one cone over a pentagon and $10$ cones over triangles.
When $b$ is in the central cell, the polytope $P_b$ is itself a pentagon.
When $b$ is in one of the five cells that share a facet with the central one, $P_b$ is a quadrilateral.
Finally, when $b$ is one of the five remaining cells, $P_b$ is a triangle.  The following code snippet computes the chamber complex in \texttt{Macaulay2} \cite{M2}. 
\begin{minted}{julia}
A = matrix{{1,1,1,1,1},{2,1,0,1,0},{1,2,0,0,1}}
B = {{5,-4},{-4,5},{2,2},{-6,3},{3,-6}}
F = gfanSecondaryFan B
all_fulldim_cones = cones(n,F)
all_rays = rays(F)
matrices = apply(all_fulldim_cones, s -> A*submatrix(all_rays,s))
cells_CA = apply(matrices,i->posHull(i))
\end{minted}
The list \texttt{cells\_CA} contains all cells of ${\cal C}_A$. Our computation follows \cite[Remark~21]{alexandr2023maximum}.
\end{example}

\begin{proposition} \label{prop:rational}
    Let $C \in \mathcal{C}_A$ be a cell. Let $n_C$ be the number of facets of $P_b$ for $b\in {\rm int}(C)$ and let $Q_b = B^T \cdot P_b$, for some kernel matrix $B \in \mathbb{R}^{n \times (n-d)}$ of $A$. The function $f(b,y) = \mathrm{vol}_{n-d}(Q_b-y)^\circ$ is a homogeneous rational function on 
    \[ \{(b,y) \,:\, b \in C \cap {\rm pos}(A),\, y \in {\rm int}(Q_b) \},\]
    of degree $d-n$. Its numerator has degree $d-n+n_C$ and the denominator has degree $n_C$.
\end{proposition}

\begin{proof}
We prove the statement for $b \in {\rm int}(C)$. The result extends to $b \in C \cap {\rm pos}(A)$ by continuity. The dual volume function can be expressed as follows: 
$$
f(b,y) \, = \, {\rm vol}_{n-d} \, (Q_b-y)^\circ \, = \, \gamma(b) \cdot \frac{\alpha(b,y)}{\ell_1(b,y) \cdot \cdots \cdot \ell_{n_C}(b,y)}.
$$
Here $\gamma$ is a nonzero function of $b$, $\ell_i(b,y) = 0$ is a linear equation defining the $i$-th facet hyperplane of $Q_b$, and $\alpha(b,y)$ is the adjoint polynomial of $Q_b$, see \eqref{eq:volrational}.     
The proposition will follow from analyzing these functions. 
By construction, the $\ell_i(b,y)$ can be chosen as $n_C$ of the (homogeneous) linear entries of the following vector:
\begin{equation} \label{eq:faceteqs}
\begin{pmatrix}
    A \\
    B^T
\end{pmatrix}^{-1}
\begin{pmatrix}
    b\\
    y
\end{pmatrix}.
\end{equation}
We denote these by $\ell_i(b,y)= c_i(b)+\langle w_i,y\rangle$, where $w_i \in \mathbb{R}^{n-d}$ and $c_i(b)$ are homogeneous linear forms in $b$. By Proposition \ref{prop:simple}, $Q_b$ is a simple polytope. Hence, we can apply \eqref{eq:adjoint} to compute the adjoint polynomial $\alpha(b,y)$: 
$$
\alpha(b,y) =  \sum\limits_{v \in {\cal V}(Q_b)} \left(|\det W_{I(v)}| \cdot \prod\limits_{i \not\in I(v)}\left(c_i(b) + \langle w_i,y\rangle\right)\right).
$$
Since $Q_b$ is simple, each vertex is adjacent to exactly $n-d$ facets. 
This means that, up to the prefactor, $\alpha(b,y)$ is a nonzero sum of homogeneous polynomials of degree $n_C-(n-d)$. 
    We have now determined the function ${\rm vol}_{n-d} \,(Q_b-y)^\circ$ up to an overall scaling by $\gamma(b)$. The proposition is proved once we show that $\gamma(b) \in \mathbb{R}\setminus \{0\}$ is a constant. For this, we rely on the theory of \emph{positive geometries} \cite{arkani2017positive,lam2022invitation}. Since the dual volume is the canonical function of $Q_b$ as a positive geometry \cite[Theorem 3]{lam2022invitation}, the residues of this function at the vertices of $Q_b$ must be equal to $\pm 1$ for any $b\in C$. Taking the iterated residue at $u \in {\cal V}(Q_b)$ (see \cite[Theorem 1 and Eqn (8)]{lam2022invitation}) results in
    $$
    \operatorname{res}\limits_{u} {\rm vol}_{n-d} \, (Q_b-y)^\circ \, = \, \gamma(b)\, \kappa_u \, \dfrac{\alpha(b,u)}{\prod\limits_{i \not \in I(u)}(c_i(b)+\langle w_i, u\rangle)} \, = \, \pm 1, 
     $$
   where $\kappa_u = (\det W_{I(u)})^{-1} \in \mathbb{R}\setminus \{0\}$. Using the fact that $\alpha(b,u)$ equals
   $$\alpha(b,u) \, = \, |\det W_{I(u)}| \cdot \prod\limits_{i \not\in I(u)}\left(c_i(b) + \langle w_i,u\rangle\right),$$ 
   we see that $\gamma(b) = \pm (\det W_{I(u)}/|\det W_{I(u)}|) = \pm 1$ is indeed a nonzero constant.
\end{proof}
In $x$-coordinates, the proof of Proposition \ref{prop:rational} leads to nice expressions like \eqref{eq:Vxpentagon} for the dual volume $V(x)$ from \eqref{eq:Vx}. For any $b \in {\rm int}(C)$, let ${\cal F}_C \subset [n]$ be the indices of the entries of \eqref{eq:faceteqs} which correspond to facets of $Q_b$ and, for each vertex of $Q_b$, let $I(v) \subset {\cal F}_C$ be the set of indices of facets containing $v$. These sets are independent of the choice of $b \in {\rm int}(C)$. The set of all index sets $I(v)$ records the vertices of $Q_b$ for $b \in \mathrm{int}(C)$. We denote it by ${\cal V}_C$. For an index set $I \subset [n]$, we write $x_I = \prod_{i\in I} x_i$ for the corresponding product of $x$-variables. Since $A \cdot W = 0$, we have $\det W_{I(v)}  = \pm \gamma \det A_{[n]\setminus I(v)}$ for some $\gamma \in \mathbb{R}$, which shows the following.
\begin{corollary} \label{cor:ratx}
    Let $C \in {\cal C}_A$ be a cell. The restriction of the dual volume function $V(x) = {\rm vol}_{n-d}(B^T \cdot P_{Ax} - B^T \cdot x)$ to the cone $\{x \in \mathbb{R}^n_{>0} \, : \, Ax \in C\}$ is given by 
    \[ V_C(x) \, = \, \gamma \cdot \frac{\sum_{I(v) \in {\cal V}_C} | \det A_{[n] \setminus I(v)}|  \cdot x_{{\cal F}_C \setminus I(v)}}{x_{{\cal F}_C}}\]
    for some positive constant $\gamma$ which depends on the choice of $B$. 
\end{corollary}
We conclude this section by using Proposition \ref{prop:rational} to derive the degree bound for the algebraic boundary of an important class of objects in convex geometry, the so called \emph{Santal\'o regions}. These are defined in \cite{meyer1998santalo} for an arbitrary convex body $K$ and any $a \in \mathbb{R}_{>0}$:
$$K_a:=\{x \in \mathrm{int}(K) : \mathrm{vol}(K-x)^\circ -  \mathrm{vol}(K-x^*)^\circ \leq a\},$$
where $x^*$ is the Santal\'o point of $K$. When $K$ is a polytope, the dual volume function is rational, and $K_a$ is a semi-algebraic set. 
When $K$ is simple, Proposition \ref{prop:rational} says that the algebraic boundary of each Santal\'o region has degree $\leq n_C$, the number of facets  of $K$.

\section{The Santal\'o patchwork} \label{sec:3}

As shown in Section \ref{sec:2}, the dual volume function $f(b,y) = \mathrm{vol}_{n-d}(Q_b-y)^\circ$ is a piecewise rational function in $b$ and $y$, with one piece $f_C(b,y)$ per chamber $C \in \mathcal{C}_A$.
As noted in the Introduction, for a fixed $b$ this function in strictly convex with respect to $y$ on the interior of $Q_b$, and therefore attains a unique minimum at $y^*(b)$, which is the Santal\'o point of $Q_b = B^T \cdot P_b$. The Santal\'o point $x^*(b)$ of $P_b$ is then recovered via the linear change of coordinates given in  \eqref{eq:bytox}.
In this section we introduce the \emph{Santal\'o patchwork}, a semi-algebraic set keeping track of the Santal\'o points $x^*(b)$ for all $b \in \mathrm{pos}(A)$. 

\begin{definition} \label{def:SPA}
    The \emph{Santal\'o patchwork} ${\rm SP}(A)$ of $A \in \mathbb{R}_{\geq 0}^{d \times n}$ is the image of the map $\phi: \mathrm{pos}(A) \to \mathbb{R}^n_{>0}$, which sends $b$ to the Santal\'o point $x^*(b) = \arg \min_{x \in P_b} {\rm vol}_{n-d}(Q_b - B^Tx)^\circ$.
\end{definition}

\begin{proposition}\label{prop:homeo}
    The map $\phi$ from Definition \ref{def:SPA} is a homeomorphism onto $\mathrm{SP}(A)$.  
\end{proposition}

\begin{proof}
It is convenient to work in $(b,y)$ coordinates first. Let $\Sigma(B)$ be the open cone
\[ \Sigma(B) \, = \, \left \{(b,y) \in \mathbb{R}^n \,:\,  \begin{pmatrix}
    A \\
    B^T
\end{pmatrix}^{-1}
\begin{pmatrix}
    b\\
    y
\end{pmatrix} > 0 \right \}. \]
It is clear that $\Sigma(B) \simeq \mathbb{R}^n_{>0}$ via the linear coordinate change $\left ( \begin{smallmatrix}
    A\\B^T
\end{smallmatrix}\right )$. The map $\phi$ factors as $\phi = \left ( \begin{smallmatrix}
    A\\B^T
\end{smallmatrix}\right )^{-1} \circ \psi$, where $
\psi(b) = (b, y^*(b)) \in \Sigma(B)$. It suffices to show that $\psi$ is a homeomorphism onto its image. First, we note that the restriction of $\psi$ to the interior of any cell $C\in \mathcal{C}_A$ is given by algebraic functions and is therefore continuous. Indeed, for a fixed $b\in \mathrm{int}(C)$, $y^*(b)$ minimizes the rational function $f_C(b,y)=\mathrm{vol}_{n-d}(Q_b-y)^\circ$. Let $b_0$ be a point in the Euclidean boundary $\partial C \cap {\rm pos}(A)$. By continuity of the dual volume, $f_C(b_0,y)$ is the dual volume of $Q_{b_0}-y$ for any $y\in \mathrm{int}(Q_{b_0})$. The Santal\'o point $y^*(b_0)$ is the unique minimizer of this function on $\mathrm{int}(Q_{b_0})$. Since the dual volume is strictly convex on $\mathrm{int}(Q_{b_0})$ \cite[Proof of Proposition 1(i)]{meyer1998santalo}, $y^*(b_0)$ is a non-degenerate solution to the system of algebraic equations
\begin{equation}\label{eq:fc}
    \frac{\partial_{y_i}f_C(b_0,y)}{f_C(b_0,y)} = 0, \quad \text{for }\,i=1,
    \ldots,n-d.
\end{equation}
By the Implicit Function Theorem, there exist a neighborhood $\Omega(b_0, C) \subset {\rm pos}(A)$ of $b_0$ and a unique algebraic function $y^*_C(b)$ such that $y_C^*(b_0) = y^*(b_0)$ and
\begin{equation}\label{eq:ystarc}
     \frac{\partial_{y_i}f_C(b,y^*_C(b))}{f_C(b,y^*_C(b))} = 0, \quad \text{for }\,i=1,
    \ldots,n-d \, \text{ and } b\in \Omega(b_0,C).
\end{equation} 
Moreover, being a solution of \eqref{eq:ystarc}, $y_C^*(b)$ minimizes the dual volume $\mathrm{vol}_{n-d}(Q_b-y)$ for $b\in \Omega(b_0,C)\cap C$, that is, $y^*_C(b) = y^*(b)$ for $b\in \Omega(b_0,C)\cap C$. 
Note that by construction, for two cells $C, C'\in \mathcal{C}_A$ and for $b_0 \in C \cap C' \cap {\rm pos}(A)$, we have $y^*_{C}(b_0)=y^*_{C'}(b_0) = y^*(b_0)$. 
Since $\mathrm{pos}(A)$ is covered by $C \cap \mathrm{pos}(A)$ for cells $C \in \mathcal{C}_A$, we get that $y^*(b)$ is continuous on $\mathrm{pos}(A)$. We conclude that $\psi$ is injective and continuous, so it is a homeomorphism between $\mathrm{pos}(A)$ and its image, the graph of $y^*(b)$. See Figure \ref{fig:SPpentagon} for an illustration of such a graph.
\end{proof}

We now find a description of ${\rm SP}(A)$ as a finite union of basic semi-algebraic sets, i.e., sets defined by algebraic equations and inequalities. This will imply that ${\rm SP}(A)$ is a semi-algebraic set. For $b \in {\rm int}(C)$, the Santal\'o point $x^*(b)$ is the unique positive point among the critical points of the following (equality) constrained optimization problem: 
\begin{equation} \label{eq:eqconstrained}
{\rm minimize} \,\, \log \, V_{C}(x), \quad \text{subject to} \, \, Ax = b. 
\end{equation}
Here $V_C(x)$ is the rational function in Corollary \ref{cor:ratx}. We simplify the notation by setting 
\begin{equation} \label{eq:alphaCVC}
\gamma = 1, \quad \alpha_C(x) \, = \, \sum_{I(v) \in {\cal V}_C} | \det A_{[n] \setminus I(v)}|  \cdot x_{{\cal F}_C \setminus I(v)}, \quad \text{and} \quad V_C(x) \, = \, \frac{\alpha_C(x)}{x_{{\cal F}_C}}.
\end{equation}
Recall that $x_{{\cal F}_C} = \prod_{i \in {\cal F}_C} x_i$ is the product of all variables $x_i$ which contribute a facet in the cell $C$. Note that $x_i$ contributes a facet if and only if every $b\in \mathrm{int}(C)$ is in the interior of the convex hull of all but the $i$-th column of $A$. Furthermore, $\alpha_C(x)$ depends only on $x_i, i \in {\cal F}_C$. The partial derivatives of $\log V_C$ with respect to the variables $x$ are given by
\[ \partial_{x_i}(\log V_C) \, = \, \begin{cases}
    \frac{\partial_{x_i} \alpha_C}{\alpha_C} - \frac{1}{x_i} & i \in {\cal F}_C,\\
    0 & i \in [n] \setminus {\cal F}_C.
\end{cases}\]
Here we write $\partial_{x_i}$ for $\frac{\partial}{\partial x_i}$. Applying the method of Lagrange multipliers to \eqref{eq:eqconstrained} we obtain the following set of rational function equations in the variables $x, \lambda = (\lambda_1, \ldots, \lambda_d)$:
\[ (\partial_{x_1}(\log V_C), \ldots, \partial_{x_n}(\log V_C))^T \, = \, A^T \cdot \lambda \quad \text{and} \quad Ax \, = \, b.\]
To eliminate the multipliers $\lambda$, we apply $B^T$ to the left- and righthand side of the first set of equations. Writing $B_C$ for the submatrix of $B$ whose rows are indexed by ${\cal F}_C$, we obtain 
\[ B_C^T \cdot \left ( \frac{\partial_{x_i} \alpha_C}{\alpha_C} - \frac{1}{x_i} \right )_{i \in {\cal F}_C} \, = \, 0 \quad \text{and} \quad Ax \, = \, b.\]
These equations make sense for minimizing the dual volume of $P_b$ only when $Ax = b \in C \cap {\rm pos}(A)$, and the minimizer is the unique solution in that cone. We define the \emph{Santal\'o patch} of the cell $C \in {\cal C}_A$ to be the following basic semi-algebraic set: 
\begin{equation} \label{eq:SC}
S_C \, = \, \left \{ x \in \mathbb{R}^n_{>0} \, : \, Ax \in C \cap {\rm pos}(A) \, \, \text{ and } \, \, B_C^T \cdot \left ( \frac{\partial_{x_i} \alpha_C}{\alpha_C} - \frac{1}{x_i} \right )_{i \in {\cal F}_C} \, = \, 0 \right \}.\end{equation}
Notice that the rational equations in this definition make sense, since $\alpha_C$ and the coordinate functions $x_i$ are positive on $\mathbb{R}^n_{>0}$. We now state a consequence of the proof of Proposition~\ref{prop:homeo}.
\begin{corollary} \label{cor:homeo}
    For a cell $C\in \mathcal{C}_A$, $\phi|_{C \cap {\rm pos}(A)}:C \cap {\rm pos}(A) \to S_C$ is a homeomorphism.
    In particular, the Santal\'o patchwork $\mathrm{SP}(A)$ is the union of the Santal\'o patches:
    $$\mathrm{SP}(A) \,= \bigcup_{C \in \mathcal{C}_A} S_C,$$
    where the union is taken over the cells of $\mathcal{C}_A$. 
\end{corollary}
\begin{example}[$d = 2, n = 3$] \label{ex:linesegm}
    Consider the matrix $A = \left (\begin{smallmatrix}
        2 & 1 & 0 \\
        0 & 1 & 2
    \end{smallmatrix} \right )$. The open cone ${\rm pos}(A)$ is $\mathbb{R}^2_{>0}$ and the polytope $P_b$, for $b \in {\rm pos}(A)$, is a line segment. The complex ${\cal C}_A$ has two cells: 
    \[ C_1 \, = \, \{ (b_1,b_2) \in \mathbb{R}^2_{\geq 0} \, : \, b_1 \leq b_2 \}, \quad C_2 \, = \, \{(b_1,b_2) \in \mathbb{R}^2_{\geq 0} \, : \, b_1 \geq b_2 \}.\]
    For $C_1$, we have ${\cal F}_{C_1} = \{1,2\}$ and ${\cal V}_{C_1} = \{ \{1\}, \{2\}\}$. The dual volume function is 
    \[ V_{C_1}(x_1,x_2,x_3) \, = \, \frac{\left| \begin{smallmatrix}
        1 & 0 \\ 1 & 2
    \end{smallmatrix} \right | \cdot x_2 + \left| \begin{smallmatrix}
        2 & 0 \\ 0 & 2
    \end{smallmatrix} \right | \cdot x_1 }{x_1x_2} \, = \, \frac{2x_2 + 4x_1}{x_1x_2}. \]
    Notice that $V_{C_1}$ does not depend on $x_3$, because $x_3 = 0$ does not contribute a facet to the line segment $P_b$, $b \in {\rm int}(C_1)$. Setting $B = \begin{pmatrix}1 & -2 & 1 \end{pmatrix}^T$ gives $B_{C_1}^T = \begin{pmatrix} 1 & -2 \end{pmatrix}$. We find the following inequality description of the Santal\'o patch $S_{C_1}$:
    \begin{align*} S_{C_1} \, &= \, \left \{ x \in \mathbb{R}^3_{>0} \, : \, 2x_1 + x_2 \leq x_2 + 2x_3, \, \left( \frac{4}{2x_2 +4x_1} -\frac{1}{x_1} \right) -2 \, \left( \frac{2}{2x_2 +4x_1} -\frac{1}{x_2} \right) = 0 \right \} \\
    &= \, \left \{ x \in \mathbb{R}^3_{>0} \, : \, x_1 \leq x_3, \, 2x_1 - x_2 = 0 \right \}.
    \end{align*}
    With an analogous computation we find the following data for the cell $C_2$:
    \[ V_{C_2}(x_1,x_2,x_3) \, = \, \frac{4x_3 + 2x_2}{x_2x_3}, \quad 
    S_{C_2} \, =  \left \{ x \in \mathbb{R}^3_{>0} \, : \, x_1 \geq x_3, \, 2x_3 - x_2 = 0 \right \}.
    \]
    We conclude that the Santal\'o patchwork ${\rm SP}(A)$ is the union of two $2$-dimensional cones in $\mathbb{R}^3$. The projection $A: {\rm SP}(A) \rightarrow {\rm pos}(A)$, i.e., the restriction to ${\rm SP}(A)$ of the linear map represented by $A$, is a homeomorphism, see Figure \ref{fig:santalolinear}.
    \begin{figure}
        \centering
        \includegraphics[height = 5cm]{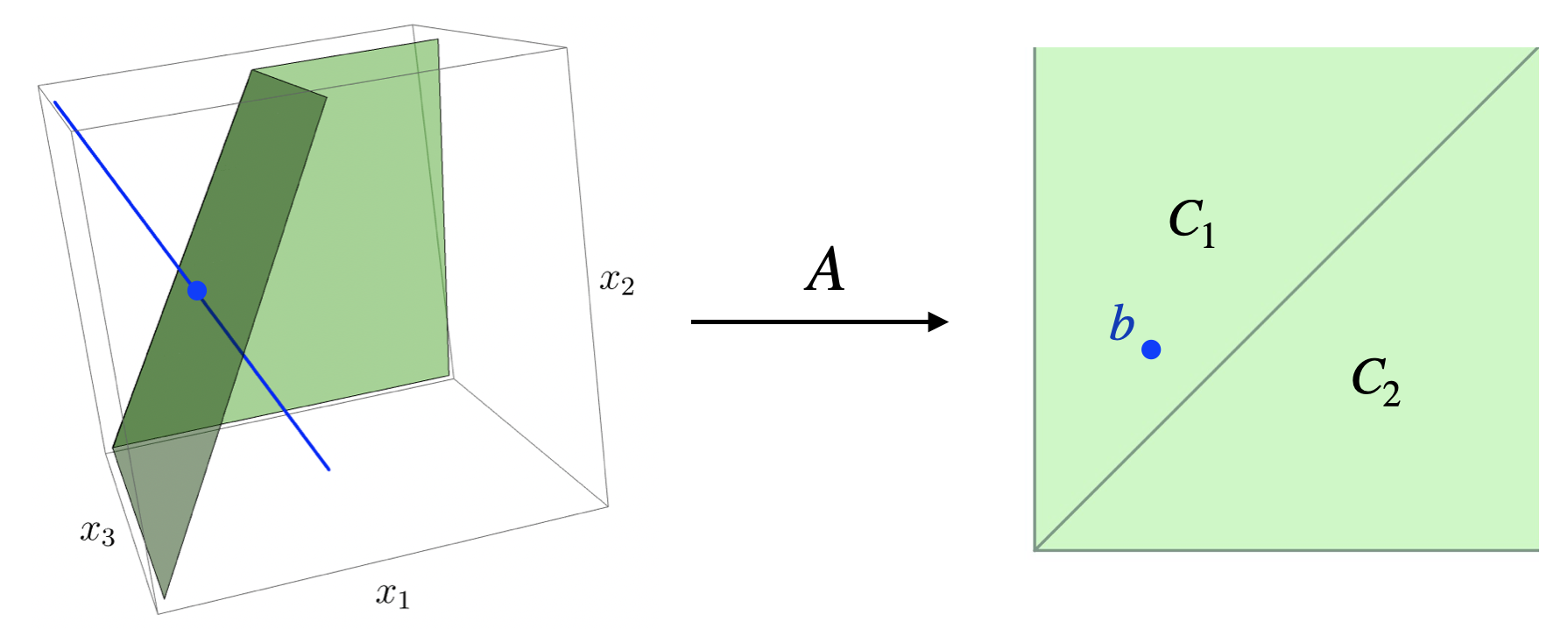}
        \caption{The Santal\'o patchwork (left) and chamber complex (right) from Example \ref{ex:linesegm}.}
        \label{fig:santalolinear}
    \end{figure}
    \end{example}

\begin{example}[$d = 2, n = 4$] \label{ex:quadrilateral}
    The chamber complex ${\cal C}_A$ for $A = \left (\begin{smallmatrix}
        1 & 1 & 1 & 1 \\
        0 & 1 & 2 & 3
    \end{smallmatrix} \right )$ has three cells: 
    \[C_1 \, = \, \{b_2 \geq 0,\,  b_1 \geq b_2 \}, \quad C_2 \, = \, \{ b_1 \leq b_2, \, 2b_1 \geq b_2 \}, \quad C_3 \, = \, \{ 2b_1 \leq b_2, \, 3b_1 \geq b_2\}. \]
    For $b \in {\rm int}(C_1)$ and $b \in {\rm int}(C_3)$, $P_b$ is a triangle, and for $b \in {\rm int}(C_2)$, it is a quadrilateral:
    \[\begin{matrix}
        {\cal F}_{C_1} \, =\,  \{2,3,4\},  \quad & {\cal V}_{C_1} \, = \, \{\{2,3\},\{2,4\},\{3,4\}\}, \\
        {\cal F}_{C_2} \, = \, \{1,2,3,4\}, \quad & {\cal V}_{C_2} \, = \, \{\{1,3\},\{1,4\},\{2,3\},\{2,4\}\}, \\ 
        {\cal F}_{C_3} \, = \, \{1,2,3\}, \quad & {\cal V}_{C_3} \, = \, \{\{1,2\},\{1,3\},\{2,3\}\}. \\
    \end{matrix}\]
    With these data, it is straightforward to write down the dual volume functions:
    \[V_{C_1} \, = \, \frac{3x_4 + 2x_3 + x_2}{x_2x_3x_4}, \quad V_{C_2} \, = \, \frac{2x_2x_4 + x_2x_3 + 3x_1x_4 + 2x_1x_3}{x_1x_2x_3x_4}, \quad V_{C_3} \, = \, \frac{x_3 + 2x_2 + 3x_1}{x_1x_2x_3}.\]
    The Santal\'o patches are 2-dimensional semi-algebraic subsets of $\mathbb{R}^4$. They are given by 
    \begin{align*}
        S_{C_1} \, &= \, \{ x >0, \, Ax \in C_1, \, 2x_3 -x_2 \, = \, 3x_4 -2x_3 \, = \, 0 \}, \\
        S_{C_2} \, &= \, \{ x >0, \,  Ax \in C_2, \, x_1x_2 - 6x_1x_4 + 2x_2x_3 + x_3x_4 =  x_1x_2-4x_1x_3+4x_2x_4-x_3x_4  = 0 \}, \\
        S_{C_3} \, &= \, \{ x > 0,\,  Ax \in C_3, \, -x_3 + 2x_2  \, = \, -2x_2 +3x_1 \,  = \, 0 \}. 
    \end{align*}
    To visualize the Santal\'o patchwork, we restrict $A: \mathbb{R}^4_{>0} \rightarrow {\rm pos}(A)$ to the probability simplex $\Delta_3 = \{ x>0, x_1 + x_2 + x_3 + x_4 = 1 \}$. The image of this restriction is the interior of the line segment obtained by taking the convex hull of the columns of $A$. The intersection ${\rm SP}(A) \cap \Delta_3$ is a piece-wise algebraic curve, homeomorphic to this line segment, see Figure \ref{fig:santaloquadrilateral}. Note the similarity between Figure \ref{fig:santaloquadrilateral} and \cite[Figure 2]{sturmfels2024toric}, where dual volume is replaced by entropy.
\end{example}
    \begin{figure}
        \centering
        \includegraphics[height = 5cm]{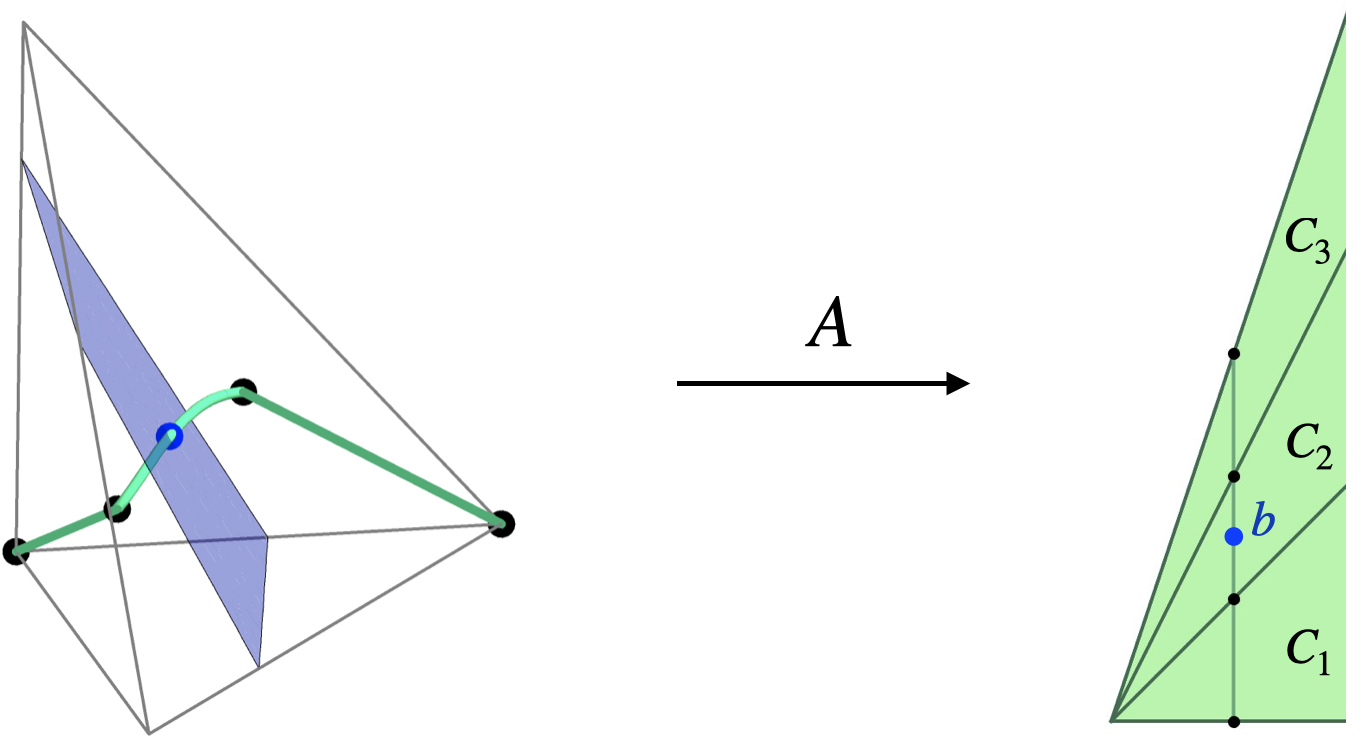}
        \caption{The Santal\'o patchwork (left) and chamber complex (right) from Example \ref{ex:quadrilateral}.}
        \label{fig:santaloquadrilateral}
    \end{figure}

\begin{example} \label{ex:runningsec3}
    The Santal\'o patchwork for the matrix $A$ in our running example (Example \ref{ex:AB}) consists of 11 patches, one for each cell in the chamber complex shown in Figure \ref{fig:pentagon1} (right). These 11 patches are separated by the black curves on the surfaces in Figure \ref{fig:SPpentagon}.
\end{example}
The following statement is a tautology. It emphasizes the role of ${\rm SP}(A)$ in solving \eqref{eq:argmin}.

\begin{proposition} \label{prop:SPcapPb}
    The Santal\'o point of $P_b$ is given by $x^*(b) = \phi(b) = {\rm SP}(A) \cap P_b$.
\end{proposition}

\begin{example}
    For $A$ as in Example \ref{ex:linesegm}, the polytope $P_b$ for $b = (1,2)$ is the blue line segment in Figure \ref{fig:santalolinear} (left). Its blue intersection point with ${\rm SP}(A)$ is the center of that line segment, which is its Santal\'o point $x^*(b)$. 
    For $A$ as in Example \ref{ex:quadrilateral}, the polytope $P_b$ for $b = (1,3/2)$ is the blue quadrilateral in Figure \ref{fig:santaloquadrilateral}. Again, ${\rm SP}(A) \cap P_b$ is the Santal\'o point. 
\end{example}

\section{Patch varieties} \label{sec:4}
Section \ref{sec:3} describes the set of all solutions to the optimization problem \eqref{eq:argmin} as a semi-algebraic set called the Santal\'o patchwork. For algebraic computations, it is often convenient to work with algebraic sets instead. This section studies algebraic varieties containing the Santal\'o patches $S_C$ defined in \eqref{eq:SC}. A natural thing to do is take the Zariski closure. 
We define
\[ X_C \, = \, \overline{S_C} \, \subset \, \mathbb{C}^{n}.  \]
We call $X_C$ the \emph{patch variety} of the cell $C$. A simple way to find equations vanishing on $X_C$ is by dropping the inequalities in \eqref{eq:SC}. Let ${\cal X}_C \subset \mathbb{C}^n$ be the Zariski closure of the set 
\[ {\cal X}_C^\circ \, = \, \left \{ x \in \mathbb{C}^n \, : \, \alpha_C(x) \prod_{i\in {\cal F}_C} x_i \neq 0 \quad \text{and} \quad B_C^T \cdot \left ( \frac{\partial_{x_i} \alpha_C}{\alpha_C} - \frac{1}{x_i} \right )_{i \in {\cal F}_C} \, = \, 0 \right \} .  \]
\begin{theorem} \label{thm:mainsec3}
    The patch variety $X_C$ is a $d$-dimensional irreducible component of ${\cal X}_C$.
\end{theorem}
\begin{proof}
    We switch to $(b,y)$-coordinates using the transformation from \eqref{eq:coordchange}. We view 
    \begin{equation} \label{eq:eqnscalX} B_C^T \cdot \left ( \frac{\partial_{x_i} \alpha_C}{\alpha_C} - \frac{1}{x_i} \right )_{i \in {\cal F}_C} \, = \, 0 \end{equation}
    as equations in $y$, parametrized by $b$. For $b_0 \in C \cap {\rm pos}(A)$, by strict convexity of $V_C(x)$, the Santal\'o point $x^*(b_0) \sim (b_0,y^*(b_0))$ corresponds to an isolated solution. By \cite[Theorem A.14.1]{sommese2005numerical}, it follows that $x^*(b_0)$ lies on a $d$-dimensional irreducible component ${\cal Z}^\circ$ of ${\cal X}_C^\circ$, and hence on an irreducible component ${\cal Z} = \overline{{\cal Z}^\circ}$ of ${\cal X}_C$. This is true for every $b_0 \in C \cap {\rm pos}(A)$, so that $S_C$ is contained in ${\cal Z}$. Hence, $X_C \subset {\cal Z}$ and it has dimension at most $d$. By Corollary \ref{cor:homeo}, $S_C$ is $d$-dimensional, so $X_C$ has dimension at least $d$. We conclude that $X_C = {\cal Z} \subset {\cal X}_C$.
\end{proof}

Notice that, by construction, the Santal\'o patch $S_C$ is stable under simultaneous scaling of the coordinates: $x^*(t\cdot b) = t\cdot x^*(b)$ for any $t \in \mathbb{R}_{>0}$. It follows that the ideal of $X_C$ can be generated by homogeneous equations. Furthermore, since the equations defining ${\cal X}^\circ_C$ are homogeneous (of degree $-1$), $I({\cal X}_C)$ is a homogeneous ideal as well. 

If $A \in \mathbb{Q}^{d \times n}_{\geq 0}$ has rational entries, the vanishing ideal $I({\cal X}_C)$ of ${\cal X}_C$ can be computed using computer algebra software such as \texttt{Macaulay2} \cite{M2} or \texttt{Oscar.jl} \cite{OSCAR} as follows. Consider the ideal of the ring $\mathbb{Q}[(x_i)_{i \in {\cal F}_C}, (\sigma_i)_{i \in {\cal F}_C}, z]$ generated by the $n - d + n_C + 1$ equations 
    \[ B_C^T \cdot \left ( z \, \partial_{x_i} \alpha_C  - \sigma_i \right )_{i \in {\cal F}_C}, \quad x_i\sigma_i-1, i \in {\cal F}_C, \quad \text{and} \quad  \, \alpha_C(x) \, z - 1.  \]
    From this ideal, eliminate the variables $\sigma_i, i \in {\cal F}_C$ and $z$. The result is $I({\cal X}_C)$.
\begin{example} \label{ex:sec3}
We perform the computation explained above for our running Example \ref{ex:AB}, for the 3-dimensional cell $C \in {\cal C}_A$ with five facets containing $b = (1,4/5,4/5)$. The adjoint~is 
\[ \alpha_C(x) \, = \, 3x_1x_2x_3 + 2x_1x_3x_5 + 2x_1x_4x_5 + 2x_2x_3x_4 + 2x_2x_4x_5, \]
i.e., the numerator of \eqref{eq:Vx}. The elimination takes place in a polynomial ring with $11$ variables. The ideal $I({\cal X}_C)$ is prime,  homogeneous, and of degree $14$. It is generated by five quintics. Here is how to compute $\alpha_C$ and $I({\cal X}_C)$ using our Julia package \texttt{Santalo.jl}, available at the online repository \url{https://mathrepo.mis.mpg.de/Santalo}:
\begin{minted}{julia}
using Santalo # load the package
A = [1 1 1 1 1; 2 1 0 1 0; 1 2 0 0 1]; b = 1//5*[5; 4; 4];
R, alpha = adjoint_x(A,b)
T, J = ideal_XC(A,b)
\end{minted}
The outputs in line 3 are the adjoint \texttt{alpha} $=\alpha_C$ and a polynomial ring \texttt{R} containing it. In line 4, we compute the ideal \texttt{J} $ = I({\cal X}_C)$ and a polynomial ring \texttt{T} containing it. 
\end{example}
Next, we ask whether ${\cal X}_C$ may fail to be equidimensional, i.e., can it have components of dimension $>d$? We do not know the answer, but we expect that for general matrices we even have $X_C = {\cal X}_C$ (see Conjecture \ref{conj:irred}). We show that the answer is \emph{no} if we perturb the objective function $V_C(x)$ slightly. More precisely, we consider the new objective function 
\begin{equation} \label{eq:VCu}
V_{C,u}(x) \, = \, \frac{\alpha_C(x)^{u_0}}{\prod_{i\in {\cal F}_C} x_i^{u_i}}.\end{equation}
Here $u_0, u_i, i \in {\cal F}_C$ are new parameters. Setting $u = {\bf 1} = (1, \ldots, 1)$ recovers our original objective function $V_C$. We will see in Section \ref{sec:4} that these new parameters have a natural statistical interpretation. The critical point equations of $\log V_{C,u}$ define the~incidence 
\[ \mathscr{X}_{C}^\circ\, = \, \left \{ (x,u) \in \mathbb{C}^n \times \mathbb{C}^{n_C+1} \, : \, \alpha_C(x) \prod_{i\in {\cal F}_C} x_i \neq 0 \quad \text{and} \quad B_C^T \cdot \left ( \frac{u_0 \partial_{x_i} \alpha_C}{\alpha_C} - \frac{u_i}{x_i} \right )_{i \in {\cal F}_C}  = \, 0 \right \}.  \]
We write $\pi_2: \mathscr{X}_{C}^\circ \rightarrow \mathbb{C}^{n_C + 1}$ for the projection $(x,u) \mapsto u$, and denote its fiber $\pi_2^{-1}(u)$ by $\mathscr{X}^\circ_{C,u}$. The variety ${\cal X}^\circ_C$ is $\mathscr{X}_{C,{\bf 1}}^\circ$. The Zariski closure of $\mathscr{X}^\circ_C$ is $\mathscr{X}_C \subset \mathbb{C}^n \times \mathbb{C}^{n_C+1}$. Fibers of $\pi_2: \mathscr{X}_C \rightarrow \mathbb{C}^{n_C+1}$ are denoted by $\mathscr{X}_{C,u}$. We have $\overline{\mathscr{X}_{C,u}^\circ} \subset \mathscr{X}_{C,u}$, and in particular ${\cal X}_C \subset \mathscr{X}_{C,{\bf 1}}$. 
\begin{proposition} \label{prop:puredimd}
    The varieties $\mathscr{X}_{C}^\circ, \mathscr{X}_{C}$ are irreducible of dimension $n_C+1+d$. There is a dense open subset $U \subset \mathbb{C}^{n_C+1}$ such that, for $u \in U$, $\mathscr{X}_{C,u}$ is pure dimensional of dimension~$d$.
\end{proposition}

\begin{proof}
    We consider the other projection $\pi_1: \mathscr{X}_C^\circ \rightarrow \mathbb{C}^n$ which sends $(x,u)$ to $x$. A fiber $\pi_1^{-1}(x)$ is defined by linear equations in $u_0, u_i, i \in {\cal F}_C$. These equations are linearly independent, because $B_C$ has rank $n-d$. This last claim follows from the fact that the rows of $B$ giving rise to $B_C$ are indexed by $\mathcal{F}_C$, which means that they contain the rays of the normal fan to a full-dimensional polytope $Q_b$ for $b\in \mathrm{int}(C)$. Hence, all fibers of $\pi_1$ are linear, and hence irreducible, of dimension $n_C+1-n+d$. By \cite[Chapter 1, \S 6, Theorem 8]{shafarevich}, $\mathscr{X}^\circ_C$ is irreducible of dimension $n_C + 1 + d$. The same holds for $\mathscr{X}_C$. Since the map $\pi_2: \mathscr{X}_C \rightarrow \mathbb{C}^{n_C+1}$ is dominant, the proposition now follows from \cite[Chapter 1, \S 6, Theorem 7]{shafarevich}.
\end{proof}

The following statement summarizes the role of these varieties in the study of the Santal\'o point of $P_b$: they provide useful semi-algebraic descriptions.

\begin{theorem} \label{thm:santalointersect}
    Let $b \in C \cap {\rm pos}(A)$ for some cell $C \in {\cal C}_A$ and let $P_b^\circ = {\rm relint}(P_b)$. The Santal\'o point $x^*(b)$ is given by 
    \[ x^*(b) \, = \, {\rm SP}(A) \cap P_b^\circ \, = \, S_C \cap P_b^\circ \, = \,  X_C \cap P_b^\circ \, = \, {\cal X}_C^\circ \cap P_b^\circ \, = \, {\cal X}_C \cap P_b^\circ \, = \, \mathscr{X}_{C,{\bf 1}}^\circ \cap P_b^\circ \, = \,  \mathscr{X}_{C,{\bf 1}} \cap P_b^\circ. \]
\end{theorem}
\begin{proof}
    The first two equalities are essentially Proposition \ref{prop:SPcapPb}. The equality $x^*(b) = {\cal X}_C^\circ \cap P_b^\circ$ follows from strict convexity of the dual volume function on $P_b$: there is only one critical point of $\log V_C$ on $P_b^\circ$. Since $({\cal X}_C \setminus {\cal X}_C^\circ) \cap \mathbb{R}^n_{> 0} = \emptyset$, replacing ${\cal X}_C^\circ$ with ${\cal X}_C$ does not change the intersection with $P_b^\circ$. The equality $x^*(b) = X_C \cap P_b^\circ$ now follows from $x^*(b) \in X_C \subset {\cal X}_C$. The last two equalities follow from $\mathscr{X}_{C,{\bf 1}}^\circ = {\cal X}_C^\circ$ and $(\mathscr{X}_{C,{\bf 1}} \setminus \mathscr{X}_{C,{\bf 1}}^\circ) \cap \mathbb{R}^n_{> 0} = \emptyset$.
\end{proof}
Next, we state a naive degree bound for the varieties defined in this section. 

\begin{proposition} \label{prop:upperbound}
    For $\bullet = X_C, {\cal X}_C$ or $\mathscr{X}_{C,u}$, for generic $u$, we have the inequality
    \[ \deg (\bullet) \leq (2n_C -n + d-1)^{n-d}. \]
\end{proposition}
\begin{proof}
   For ${\cal X}_C$, this follows from clearing denominators in \eqref{eq:eqnscalX} and applying B\'ezout's bound \cite[Chapter I, Theorem 7.7]{hartshorne2013algebraic}. For $X_C$, we use Theorem \ref{thm:mainsec3}. Finally, for $\mathscr{X}_{C,u}$, note that for generic $u \in \mathbb{C}^{n_C+1}$ we have $\mathscr{X}_{C,u} = \overline{\mathscr{X}_{C,u}^\circ}$. Adding the parameters $u$ to the equations \eqref{eq:eqnscalX} does not change the B\'ezout number.
\end{proof}

The bound from Proposition \ref{prop:upperbound} is pessimistic, e.g., for Example \ref{ex:sec3} it reads $14 \leq 49$. In particular, the varieties $X_C = {\cal X}_C$ and $\mathscr{X}_{C,u}$ have the same degree in that example. 
In the next section, we use insights from algebraic statistics to prove a lower bound on $\deg \mathscr{X}_{C,u}$ for generic $u$ (Corollary \ref{cor:degXCu}). That bound is relevant to our homotopy  method for computing Santal\'o points in Section \ref{sec:6}. Also, in experiments, we find that it approximates the actual degree more closely (Example \ref{ex:degreeVSMLdeg}). As motivated by the next conjecture, which is suggested by the examples we computed, we here mean both the degree of $X_C$ and~$\mathscr{X}_{C,u}$. 

\begin{conjecture} \label{conj:irred}
    For generic matrices $A \in \mathbb{R}^{d \times n}_{\geq 0}$ and for each cell $C \in {\cal C}_A$, there exists a dense open subset $U \subset \mathbb{C}^{n_C + 1}$ such that the variety $\mathscr{X}_{C,u}$ is irreducible of dimension $d$ for $u \in U$. Moreover, ${\bf 1} \in U$ and we have $\mathscr{X}_{C,{\bf 1}} = {\cal X}_C = X_C$.
\end{conjecture}

\section{Wachspress models} \label{sec:5}

In algebraic statistics \cite{drton2008lectures,sullivant2018algebraic}, a statistical model for a discrete random variable with $N+1$ states is the intersection of a complex algebraic variety with the probability simplex 
\[ \Delta_N \, = \, \{ (p_0, \ldots, p_N) \in \mathbb{R}^{N+1}_{>0} \, : \, p_0 + \cdots + p_N = 1 \}. \]
It is customary to identify $\Delta_N$ with the positive part $\mathbb{P}^N_{>0}$ of the $N$-dimensional complex projective space $\mathbb{P}^N$. This consists of all points in $\mathbb{R}\mathbb{P}^N$ which can be represented by positive homogeneous coordinates. A model $M$ is the intersection $M_{\mathbb{C}} \cap \mathbb{P}^N_{>0}$ of a complex projective variety $M_{\mathbb{C}}$ with the probability simplex. We require that this intersection is non-empty. For our purposes, it suffices to consider parametric models, i.e., models that come with a rational parametrization. This is true for many commonly used models, including exponential families and (conditional) independence models. Let  $p_i(y) = p_i(y_1, \ldots, y_m), i = 0, \ldots, N$ be rational functions of $m < N$ variables such that $\sum_{i=0}^N p_i(y) = 1$. The variety $M_{\mathbb{C}}$ is the closure of the image of the rational map $\mathbb{C}^m \dashrightarrow \mathbb{P}^{N}$ given by $y \mapsto (p_0(y): \ldots: p_N(y))$. \emph{Maximum likelihood estimation} for the model $M$ means finding the probability distribution $p \in M$ which makes an experimental observation $u \in \mathbb{N}^{N+1}$ most likely. More precisely, suppose that state $i$ was observed $u_i$ times in an experiment. One maximizes the \emph{log-likelihood function} 
\[ L_u \, = \, \log \frac{p_0^{u_0} p_1^{u_1} \cdots p_N^{u_N}}{(p_0 + p_1 + \cdots + p_N)^{u_0 + u_1 + \cdots + u_N}}\]
subject to the constraint $p \in M$. To study this problem algebraically, one often relaxes it to finding all complex critical points of $L_u$ on an open subset of $M_{\mathbb{C}}$. In our parametric setting, we solve the system of rational function equations 
\begin{equation} \label{eq:likelihoodeq}
\partial_{y_i} L_u(y) \, = \, \sum_{j=0}^N u_j \, \frac{\partial_{y_i}p_j(y)}{p_j(y)} = \, 0, \quad i = 1, \ldots m \quad \text{for} \quad y \in \mathbb{C}^m \setminus D. \end{equation}
Here $D \subset \mathbb{C}^m$ is the union of the supports of the divisors ${\rm div}(p_j), j = 0, \ldots, N$. That is, it is the union of all hypersurfaces in $\mathbb{C}^m$ along which one of the $p_j$ has a zero or a pole. We refer to these equations as the \emph{likelihood equations} for the model $M$. The number of complex solutions $y \in \mathbb{C}^m \setminus D$ for generic, complex data $u \in \mathbb{C}^{N+1}$ is an invariant called the \emph{maximum likelihood degree} (ML degree) of $M_{\mathbb{C}}$ \cite{catanese2006maximum}, which we denote by ${\rm MLdeg}(M_{\mathbb{C}})$. This assumes that the parametrization map $\mathbb{C}^m \setminus D \rightarrow M_{\mathbb{C}}$ given by $y \mapsto (p_0(y): \ldots: p_N(y))$ is~birational.  

The models that are relevant to our story are called \emph{Wachspress models}. These are associated to a simple polytope $Q \subset \mathbb{R}^m$, and the number of states equals the number of vertices $|{\cal V}(Q)|$. We use the notation \eqref{eq:minfacets} for the face description of $Q$. The parametrizing functions of our model are naturally obtained from the formula \eqref{eq:adjoint} for the adjoint: 
\begin{equation} \label{eq:wachspressparam}
p_v(y) \, = \, \frac{|\det W_{I(v)}| \cdot \prod_{i\notin I(v)}(c_i + \langle w_i, y \rangle)}{\alpha_Q(y)}, \quad v \in {\cal V}(Q).\end{equation}
This gives a rational map $\mathbb{C}^m \dashrightarrow \mathbb{P}^{|{\cal V}(Q)|-1}$, whose image closure $M_{\mathbb{C}}(Q)$ is the \emph{Wachspress variety} of $Q$. Note that the coordinates $p_v$ for $v \in {\cal V}(Q)$ sum to 1 by construction. These varieties appear in the context of geometric modelling \cite{garcia2010linear}, and Wachspress surfaces were studied in \cite{irving2014geometry}. To the best of our knowledge, the interpretation as a statistical model first appeared in \cite{kohn2020moment}. Bayesian integrals for these models were studied in \cite{borinsky2023bayesian}. The divisor $D$ from \eqref{eq:likelihoodeq} for the Wachspress model $M_{\mathbb{C}}(Q)$ is the union of the adjoint hypersurface $\alpha_Q(y) = 0$ and the facet hyperplanes $c_i + \langle w_i,y \rangle = 0$. We denote this hypersurface by $D(Q)$. 

\begin{lemma} \label{lem:isom}
    Let $Q \subset \mathbb{R}^m$ be a simple polytope with Wachspress model $M_{\mathbb{C}}(Q)$. Let $N = |{\cal V}(Q)|-1$ and define the divisor ${\cal H} \subset \mathbb{P}^{N}$ given by $(\sum_{v \in {\cal V}(Q)} p_v)\prod_{v \in {\cal V}(Q)}p_v = 0$. Here $(p_v)_{v \in {\cal V}(Q)}$ are homogeneous coordinates on $\mathbb{P}^N$. The map $\psi\!:\!\mathbb{C}^m \setminus D(Q) \rightarrow M_{\mathbb{C}}(Q) \setminus {\cal H}$ given by $y \mapsto (p_v(y))_{v \in {\cal V}(Q)}$, with $p_v(y)$ from \eqref{eq:wachspressparam}, is an isomorphism. 
\end{lemma}
\begin{proof}
    First note that the morphism $\psi$ is well-defined. The functions $p_v$ are regular on $\mathbb{C}^m \setminus D(Q)$, and the image of $\mathbb{C}^m \setminus D(Q)$ is contained in the complement of ${\cal H}$. 

    It remains to show that $\psi$ is invertible. Consider an automorphism $\varphi$ of $\mathbb{P}^N$ defined by $p_v \mapsto \frac{p_v}{|\det W_{I(v)}|}$. The map $\varphi \circ \psi$ is a restriction of the \emph{Wachspress map} defined in \cite[Equation (5)]{kohn2020projective}, which is invertible by \cite[Theorem 4]{kohn2020projective}. It follows that $\psi$ is invertible too. 
\end{proof}

\begin{corollary} \label{cor:eulerchar}
    The maximum likelihood degree ${\rm MLdeg}(M_{\mathbb{C}}(Q))$ of the Wachspress model of $Q$ equals the absolute value of the Euler characteristic $\chi(\mathbb{C}^m \setminus D(Q)) = \chi(M_{\mathbb{C}}(Q) \setminus {\cal H})$. 
\end{corollary}
\begin{proof}
    By Lemma \ref{lem:isom}, $M_{\mathbb{C}}(Q) \setminus {\cal H}$ is smooth, so \cite[Theorem~1.7]{huh2014likelihood} implies the statement.
\end{proof}
Solving the likelihood equations of $M_{\mathbb{C}}(Q)$ with data $u$ is equivalent to computing the intersection of the fiber $\mathscr{X}_{C,u'}^\circ$, defined in Section \ref{sec:3}, with a linear space. The parameters $u' \in \mathbb{C}^{n_C +1}$ are obtained from $u$ via a linear map. This is the content of our next theorem. 
\begin{theorem} \label{thm:MLEvsOPT}
    Let $Q_{b_0} = B^T \cdot P_{b_0}$ for $b_0 \in C \cap {\rm pos}(A)$, where $C \in {\cal C}_A$ is a cell in the chamber complex of $A$. The complex critical points of the log-likelihood function $L_u(y)$ for the Wachspress model $M_{\mathbb{C}}(Q_{b_0})$ with data $u = (u_v)_{v \in {\cal V}(Q_{b_0})}$ are in one-to-one correspondence with the complex critical points of $V_{C,u'}(x)$ from \eqref{eq:VCu} on  $\{Ax = {b_0}\}$, where $u'$ has entries 
    \begin{equation} \label{eq:uprime} u_0' \, = \, - \sum_{v \in {\cal V}(Q_b)} u_v, \quad u_i' \,  = \, - \sum_{v:i \notin I(v)} u_v , \, i \in {\cal F}_C. \end{equation}
    More precisely, the critical points $y_{\rm crit} \in \mathbb{C}^{n-d} \setminus D(Q_{b_0})$ of $L_u(y)$ are $B^T \cdot x_{\rm crit}$, where $x_{\rm crit}$ ranges over the points in the intersection 
    $\mathscr{X}_{C,u'}^\circ \cap \{ A \, x = {b_0}\}$, 
\end{theorem}
\begin{proof}
    The likelihood function for the data $u = (u_v)_{v \in {\cal V}(Q_{b_0})}$ is 
    \[ {\rm exp} \,  L_u(y)  =  \prod_{v \in {\cal V}(Q_{b_0})} \left (  \frac{|\det W_{I(v)}| \cdot \prod_{i \in {\cal F}_C \setminus I(v)}(c_i(b_0) + \langle w_i, y \rangle)}{\alpha_{Q_{b_0}}(y)} \right )^{u_v} =  V_{C,u'}\left ( \begin{pmatrix}
        A \\ B^T
    \end{pmatrix}^{-1} \begin{pmatrix}
        b_0 \\ y
    \end{pmatrix}\right ). \]
    This uses the change of coordinates \eqref{eq:coordchange}: $x_i = c_i(b_0) + \langle w_i, y \rangle$. Applying the chain rule gives 
    \[ \partial_{y_j} \log V_{C,u'} \, = \, \sum_{i \in {\cal F}_C} \partial_{x_i} \log V_{C,u'} \cdot \frac{{\rm d}x_i}{{\rm d}y_j}.\]
    It follows that the likelihood equations for the Wachspress model $M_{\mathbb{C}}(Q_{b_0})$ are equivalent to
    \[ W_C^T\cdot \left ( \frac{u_0' \partial_{x_i} \alpha_C}{\alpha_C} - \frac{u_i'}{x_i} \right )_{i \in {\cal F}_C}  = \, 0 \quad \text{and} \quad Ax \, = \, b_0. \]
    Here $Ax = b_0$ is $b = b_0$ in $x$-coordinates, and $W_C$ is the submatrix of the matrix $W$ of facet normals whose rows are indexed by ${\cal F}_C$. The column span of $W_C$ equals that of $B_C$ by construction, so these are precisely the equations for $\mathscr{X}_{C,u'}^\circ \cap \{Ax = b_0\}$.
\end{proof}
Our next statement uses the following definition. An isolated solution $x$ to the $n$~equations 
\begin{equation} \label{eq:ueqs}
B_C^T \cdot \left ( \frac{\tilde{u}_0 \partial_{x_i} \alpha_C}{\alpha_C} - \frac{\tilde{u}_i}{x_i} \right )_{i \in {\cal F}_C}  = \, 0 \quad \text{and} \quad Ax = b_0\end{equation}
for fixed $\tilde{u} \in \mathbb{C}^{n_C+1}$ is \emph{regular} if the rank of the Jacobian matrix at $x$ is $n$.
\begin{proposition} \label{prop:Iub}
    Let $b_0 \in C \cap {\rm pos}(A)$ for some cell $C\in {\cal C}_A$. There is a dense open subset $U \subset \mathbb{C}^{n_C + 1}$ such that, for $\tilde{u} \in U$, the set $I(\tilde{u},b_0) := \mathscr{X}_{C,\tilde{u}}^\circ \cap \{ A x = b_0 \}$ consists of ${\rm MLdeg}(M_{\mathbb{C}}(Q_{b_0}))$ regular points. Moreover, the number of regular isolated points in $I(\tilde{u},b_0)$ for any $\tilde{u}$ is at most ${\rm MLdeg}(M_{\mathbb{C}}(Q_{b_0}))$. 
\end{proposition}
\begin{proof}
The data points $u = (u_v)_{v \in {\cal V}(Q_{b_0})}$ for the Wachspress model $M_{\mathbb{C}}(Q_{b_0})$ parametrize a linear subspace $H$ of $\mathbb{C}^{n_C +1}$ via \eqref{eq:uprime}. 
    By Theorem \ref{thm:MLEvsOPT} and the definition of the ML degree, the number of points in $I(u',b_0)$ is ${\rm MLdeg}(M_{\mathbb{C}}(Q_{b_0}))$ for generic $u' \in H$. By Corollary \ref{cor:eulerchar}, this number equals the signed Euler characteristic of $\mathbb{C}^m \setminus D(Q_{b_0})$. By \cite[Theorem 1]{huh2013maximum}, that Euler characteristic is the number of regular complex critical points of \[ V_{C,\tilde{u}} \left ( \begin{pmatrix}
        A \\ B^T
    \end{pmatrix}^{-1} \begin{pmatrix}
        b_0 \\ y
    \end{pmatrix}\right ) \, = \, \frac{\alpha_{Q_{b_0}}(y)^{\tilde{u}_0}}{\prod_{i \in {\cal F}_C} (c_i(b_0) + \langle w_i , y \rangle)^{\tilde{u}_i}}. \] 
    for generic $\tilde{u} \in \mathbb{C}^{n_C+1}$. The final statement about the upper bound follows from the fact that the generic number of regular isolated solutions to the system of equations \eqref{eq:ueqs}
    equals the maximal number of regular isolated solutions, see for instance \cite[Theorem 7.1.1]{sommese2005numerical}.
\end{proof}

\begin{corollary} \label{cor:degXCu}
    For any cell $C \in {\cal C}_A$ and generic $\tilde{u} \in \mathbb{C}^{n_C+1}$ the degree of the variety $\mathscr{X}_{C,\tilde{u}} \subset \mathbb{C}^n$ is at least  ${\rm MLdeg}(M_{\mathbb{C}}(Q_{b_0}))$, where $b_0$ is a generic point in ${\rm int}(C)$. 
\end{corollary}
\begin{proof}
    By Proposition \ref{prop:puredimd}, $\mathscr{X}_{C,\tilde{u}}$ is pure dimensional of dimension $d$ for generic $u$. Its degree is the maximal number of regular intersection points with a linear space of codimension $d$. This is at least the cardinality of $I(\tilde{u},b_0)$. The statement is a consequence of Proposition~\ref{prop:Iub}. 
\end{proof}

Though the Santal\'o point of $Q_{b_0}$ is one of the regular intersection points in $I({\bf 1},b_0) = \mathscr{X}^\circ_{C,\bf{1}} \cap \{Ax = b_0\}$ (Theorem \ref{thm:santalointersect}), the usefulness of the results in this section for our original problem may seem somewhat mysterious. It will become clear in Section \ref{sec:6} that Proposition \ref{prop:Iub} is crucial for our homotopy continuation based algorithm for computing Santal\'o points. 

\begin{remark} \label{rem:optvsmle}
    Dual volume minimization is not the only convex optimization problem on $P_b$ that has the interpretation of a maximum likelihood estimation problem. Other commonly used objective functions lead to maximum likelihood estimation for different models. We briefly discuss the cases $V(x) = - \sum_{i=1}^n \log x_i$ (log-barrier) and $V(x) = \sum_{i=1}^n x_i \log x_i - x_i$ (entropic regularization) mentioned in the Introduction. In each case, there are $N+1 = n$ states. For ease of exposition, we make some additional assumptions on the matrix $A$. 
    
    First, for $V(x) = - \sum_{i=1}^n \log x_i$, assume that the entries of each column of $A$ sum to the same number $c$. 
    The statistical model $M$ in this context is the linear model obtained by intersecting the row span $M_{\mathbb{C}}$ of $A$ with $\Delta_{n-1}$. It is parametrized by $p_i(y) = (y^Ta_i)/(y^TA {\bf 1})$, where $a_i$ is the $i$-th column of $A$ and ${\bf 1} \in \mathbb{R}^n$ is the all-ones vector. One checks that the maximum likelihood estimate for the data $u = (1, \ldots, 1)$ is the unique positive minimizer of the log-barrier function $V(x)$ on the affine-linear space $\{Ax = b\}$, where $b = c^{-1} \,  A \, {\bf 1}$.

    For $V(x) = \sum_{i=1}^n x_i \log x_i - x_i$, the model comes from a toric variety. We assume that the first row of $A$ is the all-ones vector ${\bf 1}$ and write $a_i \in \mathbb{R}^d_{\geq 0}$ for the $i$-th column. These columns define a monomial map, whose image is $M_{\mathbb{C}}$. Concretely, let $f(y) = y^{a_1} + y^{a_2} + \cdots + y^{a_n}$ and consider the rational parametrization functions $p_i(y) = y^{a_i}/f(x)$, parametrizing $M_{\mathbb{C}}$. For any data vector $u = (u_1, \ldots, u_n) \in \mathbb{N}^n$, let $\bar{u} = (\sum_{i=1}^n u_i)^{-1}\cdot u$ be the \emph{empirical distribution}. As a consequence of Birch's theorem \cite[Proposition 2.1.5]{drton2008lectures}, the maximum likelihood estimate for the model $M$ is the unique positive minimizer of the entropy $V(x)$ on $\{Ax = A \bar{u}\}$.       
\end{remark}

There is no explicit formula for the maximum likelihood degree of the Wachspress model $M_{\mathbb{C}}(Q)$. We end the section with a conjecture for polygons in the plane. We represent a \emph{generic $n$-gon} by a fiber $P_b$ of $A : \mathbb{R}^n_{\geq 0} \rightarrow {\rm cone}(A)$, where $A \in \mathbb{R}^{(n-2) \times n}_{\geq 0}$ is generic among those matrices for which there is a cell in ${\cal C}_{A}$ whose fibers are $n$-gons. Concretely, let
\[C_{\max} \, = \, {\rm cone}(A_{[n] \setminus \{1\}}) \cap {\rm cone}(A_{[n] \setminus \{2\}})  \cap \cdots  \cap {\rm cone}(A_{[n] \setminus \{n\}}) \, \neq \, \emptyset \quad \text{and} \quad \dim(C_{\max}) = n-2. \]
This uses the notation introduced before Definition \ref{def:chambercomplex}. In general, $C_{\max}$ is a union of cells in ${\cal C}_A$. We pick $b \in {\rm int}(C)$ for any cell $C \subset C_{\max}$. Our predicted formula for the ML degree of the Wachspress model of a generic $n$-gon rests on the following conjecture. 
\begin{conjecture} \label{conj:genericadjoint}
    For a generic $n$-gon $Q = B^T \cdot P_b$, the adjoint $\alpha_Q$ is non-degenerate in the sense of \cite[Theorem 3]{huh2013maximum}, with Newton polytope equal to that of $(1+ y_1 + y_2)^{n-3}$. 
\end{conjecture}

\begin{proposition} \label{prop:MLdeg}
    If Conjecture \ref{conj:genericadjoint} holds, then the maximum likelihood degree of the Wachspress model of a generic $n$-gon $Q = B^T \cdot P_b$ is 
    \[ {\rm MLdeg}(M_{\mathbb{C}}(Q)) \,= \, (n-1)(n-2) + (n-3)(n-5) - 1 . \]
\end{proposition}

\begin{proof}
    By Corollary \ref{cor:eulerchar}, we have ${\rm MLdeg}(M_{\mathbb{C}}(Q)) = \chi(\mathbb{C}^2 \setminus D(Q))$, where $D_{Q}$ is the curve $\{\alpha_Q(y) \prod_{i=1}^n l_i(y) = 0\}$. Here we write $l_i(y) = c_i + \langle w_i,y \rangle$ for the equations of the lines defining the edges of $Q$. The excision property of the Euler characteristic gives
    \[ \chi(\mathbb{C}^2 \setminus D(Q)) \, = \, \chi\big (\mathbb{C}^2 \setminus \big \{ \prod_{i=1}^n l_i(y) = 0 \big \}\big ) \, - \, \chi \big (\{\alpha_Q(y) = 0 \} \setminus \big \{ \prod_{i=1}^n l_i(y) = 0 \big \} \big ).  \]
    Since the line arrangement of $l_1, \ldots, l_n$ is generic, the first term is $\binom{n-1}{2}$ \cite[Equation (8)]{hosten2005solving}. On the second term, we use excision once more: 
    \[ \chi \big (\{\alpha_Q(y) = 0 \} \setminus \big \{ \prod_{i=1}^n l_i(y) = 0 \big \} \big ) \, =\, \chi(\{\alpha_Q(y) = 0\}) - \chi \big( \{ \alpha_Q = 0 \} \cap \big \{ \prod_{i=1}^n l_i(y) = 0 \big \} \big).\]
    Here the second term is $-\binom{n-1}{2}-1$, the number of \emph{residual points} of $Q$ \cite[Section 2.1]{kohn2021adjoints}. What's missing is the Euler characteristic of the affine curve $\chi(\{\alpha_Q(y) =0\})$. Assuming Conjecture \ref{conj:genericadjoint}, this Euler characteristic equals  $-(n-3)^2 + 2(n-3)$. Summing all this up gives the formula in the proposition.
\end{proof}

In the spirit of Corollary \ref{cor:degXCu}, we can compare the number $(n-1)(n-2) + (n-3)(n-5) - 1$ to the degree of the variety $\mathscr{X}_{C,u}$, and hence that of ${\cal X}_C$ and $X_C$ (Conjecture \ref{conj:irred}). 

\begin{example} \label{ex:degreeVSMLdeg}
    For $n = 3, 4, \ldots, 11$ we generate a totally positive matrix $A \in (\mathbb{R})^{(n-2)\times n}_{\geq 0}$ (meaning that all $(n-2)$-minors are positive) and we pick a cell $C \subset C_{\max}$. Using the numerical homotopy continuation techniques discussed in the next section, we compute that 
    \begin{center}
    \begin{tabular}{c|ccccccccc}
         $n$ &  3 & 4 & 5 & 6 & 7 & 8 & 9 & 10 & 11 \\ \hline
        $(n-1)(n-2) + (n-3)(n-5) - 1$ & 1 & 4 & 11 & 22 & 37 & 56 & 79 & 106 & 137\\
         $\deg(\mathscr{X}_{C,u}) = \deg({\cal X}_C) = \deg (X_C) $ & 1 & 4 & 14 & 27& 44 & 65 & 90 & 119 & 152
    \end{tabular}
    \end{center}
    For instance, for $n= 5$, a generic linear space $\{ \tilde{A} x = \tilde{b}\}$ of dimension 2 intersects $\mathscr{X}_{C,u}$ in 14 points. By Proposition \ref{prop:Iub}, the \emph{special} linear space $\{ A x = b\}$ leads to only 11 points. Hence, the lower bound in Corollary \ref{cor:degXCu} may be strict. The table leads us to conjecture that
    \[ \deg(X_C) \, = \, (n-1)(n-2) + (n-3)(n-5) - 1 + 2(n-3) - 1, \quad \text{ for $n \geq 5$}. \]
    Code for computing these degrees is found at \cite{mathrepo}.
\end{example}

\section{Computing Santal\'o points} \label{sec:6}
We discuss how to compute Santal\'o points numerically. We consider two different situations. First, the input is a polytope $Q \subset \mathbb{R}^m$, and the output is its Santal\'o point $y^*$ from \eqref{eq:ystar}. Our continuation algorithm exploits the likelihood geometry from Section \ref{sec:5}. The second scenario computes the Santal\'o point $x^*(b_1)$ from $x^*(b_0)$, assuming $b_1$ lies in the same cell $C \in {\cal C}_A$ as $b_0$. The strategy here is to track a real path on the Santal\'o patch $S_C$. These algorithms are implemented in Julia (v1.9.1) using \texttt{Oscar.jl} (v0.14.0) \cite{OSCAR} and \texttt{HomotopyContinuation.jl} (v2.9.3) \cite{breiding2018homotopycontinuation}. All code is available at \cite{mathrepo}.

The computational paradigm behind our algorithms is that of \emph{homotopy continuation}. We briefly recall the main idea and refer to the standard textbook  \cite{sommese2005numerical} for more details. Let $F: (\mathbb{C}^n \setminus D) \times \mathbb{C}^m \rightarrow \mathbb{C}^n$ be a map whose coordinates are rational functions in $x = (x_1, \ldots, x_n)$, depending polynomially on $m$ parameters $q \in \mathbb{C}^m$. We assume that the denominators of these functions do not depend on $q$, and their vanishing locus is contained in the hypersurface $D \subset \mathbb{C}^n$, so that $F$ is a regular map. We consider the incidence variety 
\[ Y \, = \, F^{-1}(0) \, = \, \{ (x,q) \in (\mathbb{C}^n \setminus D) \times \mathbb{C}^m \, : \, F(x,q) = 0 \}. \]
A fiber of the natural projection $\pi_q: Y \rightarrow \mathbb{C}^m$ is denoted by $Y_{q_0} = \pi_q^{-1}(q_0)$. It consists of all solutions to the $n$ equations in $n$ variables $F(x,q_0) = 0$. A solution $(x_0,q_0)$ in $Y_{q_0}$ is called \emph{isolated} and \emph{regular} if the Jacobian of $F$ (with respect to $x$) evaluated at $(x_0,q_0)$ is an invertible $n\times n$-matrix. Typically, one has computed all isolated regular solutions in $Y_{q_0}$ and is interested in computing those in $Y_{q_1}$, for some parameters $q_0 \neq q_1 \in \mathbb{C}^m$. Homotopy continuation rests on the \emph{parameter continuation theorem} \cite[Theorem 7.1.4]{sommese2005numerical}. First, this states that the number of isolated regular solutions in $Y_{q_0}$ is constant for $q_0 \in \mathbb{C}^m \setminus \nabla$, where $\nabla \subset \mathbb{C}^m$ is a proper subvariety. Second, let $\gamma: [0,1] \rightarrow \mathbb{C}^m$ be a continuous path such that $\gamma(0) = q_0$, $\gamma(1) = q_1$ and $\gamma([0,1)) \cap \nabla = \emptyset$. Since $q_0 \notin \nabla$, each isolated regular solution $(x_0,q_0) \in Y_{q_0}$ defines a unique smooth solution path $(t, x(t))$ satisfying 
\[F(x(t),\gamma(t)) \, = \, 0, \quad t \in [0,1), \quad x(0) = x_0 .\]
Moreover, the limits of all these solution paths as $t \rightarrow 1$ contain all isolated regular solutions in $Y_{q_1}$. Numerical path trackers, such as that implemented in \texttt{HomotopyContinuation.jl}, track these solution paths numerically for $t$ going from $0$ to $1$. For obvious reasons, the system of equations $F(x,q_0)=0$ is called the \emph{start system}, and $F(x,q_1)=0$ is the \emph{target~system}. 

A useful algorithm for finding all isolated regular solutions in $Y_{q_0}$, i.e., the solutions to the start system, is itself based on homotopy continuation. It uses \emph{monodromy loops} \cite{duff2019solving}. The method needs the assumption that one solution $(x_0,q_0) \in Y_{q_0}$ is known. One chooses $\gamma$ to be a closed path, i.e., $\gamma(0) = \gamma(1) = q_0$. If this path encircles a ramification point of the branched cover $\pi_q: Y \rightarrow \mathbb{C}^m$, then the corresponding solution path $(t,x(t))$ may provide a new regular isolated solution in $Y_{q_0}$: $x(1) \neq x(0)$. If $Y$ is irreducible, all isolated regular solutions can be found by repeating this process \cite[Remark 2.2]{duff2019solving}. To know when enough loops are tracked, it is very useful to compute the maximal number of solutions from a theoretical argument. This is one of the purposes of Proposition \ref{prop:Iub} and Proposition \ref{prop:MLdeg}. The monodromy method, and in particular its implementation in the command \texttt{monodromy\_solve} in \texttt{HomotopyContinuation.jl}, is very efficient and reliable in practice.

\subsection{From likelihood equations to dual volume}
Let $Q \subset \mathbb{R}^m$ be a full-dimensional simple polytope with minimal facet representation 
\[ Q \, = \, \{ y \in \mathbb{R}^m \, : \, W \, y + c \geq 0 \}, \quad \text{for } W \in \mathbb{R}^{n \times m}, \, c \in \mathbb{R}^n.\] 
Let $A \in \mathbb{R}^{d \times n}$ be a cokernel matrix of $W$ (so that $A \cdot W = 0$). Here $d = n-m$, and $A$ can be chosen so that its entries are nonnegative. Setting $x = W \, y + c$, we see that $Q$ is a projection of $P_b = \{x \in \mathbb{R}^n_{\geq 0} \, : \, Ax = b\}$, with $b = Ac$. More precisely, $Q$ is given by $W^\dag \cdot (P_b-c)$, where $W^\dag \in \mathbb{R}^{m \times n}$ is the pseudo-inverse of $W$. Though we assumed nonnegative entries to guarantee compact fibers of $A: \mathbb{R}^n_{\geq 0} \rightarrow {\rm pos}(A)$, it is not necessary to find a nonnegative representation for computing the Santal\'o point. We think of the likelihood equations \eqref{eq:ueqs} as a system of equations with variables $x_1, \ldots, x_n$ and parameters $q = (u_0, (u_i)_{i \in {\cal F}_C})$: 
\begin{equation} \label{eq:homotopy}
F(x;q) \, = \, \left( \begin{matrix} B_C^T \cdot \left ( \frac{u_0 \partial_{x_i} \alpha_C}{\alpha_C} - \frac{u_i}{x_i} \right )_{i \in {\cal F}_C} \\
Ax - b \end{matrix} \right ) \, = \, 0.\end{equation}
In order to solve this for generic parameters $q_0 \in \mathbb{C}^{n_C+1}$ using monodromy loops, we need to compute one regular solution in $Y_{q_0}$. This is done as follows. Select a random point $x_0 \in \mathbb{C}^n$ so that $Ax_0 = b$ and solve the linear system $F(x_0;q)$ for $q$. We can pick any solution to these linear equations as the start parameters $q_0$. Since $Y = \mathscr{X}_{C}^\circ$ is irreducible, see Proposition \ref{prop:puredimd}, all other solutions to $F(x,q_0)$ can be found using monodromy loops. By Proposition \ref{prop:Iub}, the number of solutions is the maximum likelihood degree of the Wachspress model $M_{\mathbb{C}}(Q)$.

Once we have computed $Y_{q_0}$, we set $\gamma(t) = (1-t) \cdot q_0 + t \cdot \mathbf{1}$ and track the ${\rm MLdeg}(M_{\mathbb{C}}(Q))$-many solution paths for $t \in [0,1]$. Precisely one of the end points is positive. Indeed, the regular isolated solutions for $q_1 = \mathbf{1}$ are critical points of the logarithm of the dual volume function on $Q$. Among them, the Santal\'o point is the unique positive point, by convexity. 

\begin{example} \label{ex:6-1}
We illustrate the code on our running example using the data in \eqref{eq:Abpentagon}:
\begin{minted}{julia}
using Santalo # load the package
A = [1 1 1 1 1; 2 1 0 1 0; 1 2 0 0 1] 
B = transpose(1//18*[5 -4 2 -6 3; -4 5 2 3 -6])
b = 1//5*[5; 4; 4]
Q = compute_Q(A,b,B) # Q = B^T*Pb
ystar = get_santalo_point(Q) # Santalo point in y-coordinates
xstar = get_santalo_point(A,b) # Santalo point in x-coordinates
\end{minted}
The result $y^*$ is as reported in Example \ref{ex:pentagon1}, and $x^* \approx (0.197,0.197,0.188,0.210,0.210)$. 
\end{example}

\begin{example} \label{ex:permutahedron}
The user can also construct a polytope \texttt{Q} using the functionalities of \texttt{Oscar.jl} and use it as input for the function \texttt{get\_santalo\_point}. As a 3D example, we consider the permutahedron; a simple polytope with $f$-vector $(24,36,14)$. 
\begin{minted}{julia}
using Oscar # load the package Oscar to construct polytopes
Q = project_full(permutahedron(3))
ystar = get_santalo_point(Q) # output: (2.5, 2.5, 2.5)
\end{minted}
Here $Q$ is the convex hull of all points $(j,k,l)$, where $(i,j,k,l) \in S_4$ is a permutation of $(1,2,3,4)$. This  permutahedron is represented by the following values for $A$ and $b$:
\setcounter{MaxMatrixCols}{15}
\[ A \, = \, \begin{pmatrix} 
 1 & 0 & 0 & 1 & 0 & 0 & 0 & 0 & 0 & 0 & 0 & 0 & 0 & 0  \\ 
 1 & 1 & 0 & 0 & 1 & 1 & 0 & 0 & 0 & 0 & 0 & 0 & 0 & 0  \\ 
 0 & 0 & 1 & 0 & 0 & 1 & 0 & 0 & 0 & 0 & 0 & 0 & 0 & 0  \\ 
 0 & 0 & 1 & 1 & 0 & 0 & 1 & 0 & 0 & 0 & 0 & 0 & 0 & 0  \\ 
 1 & 0 & 0 & 0 & 0 & 0 & 0 & 1 & 0 & 1 & 0 & 0 & 0 & 0  \\ 
 0 & 0 & 1 & 0 & 0 & 0 & 0 & 0 & 1 & 1 & 0 & 0 & 0 & 0  \\ 
 0 & 1 & 0 & 0 & 0 & 0 & 0 & 0 & 0 & 1 & 0 & 0 & 0 & 0  \\ 
 0 & 1 & 0 & 1 & 0 & 0 & 0 & 0 & 0 & 0 & 1 & 0 & 0 & 0  \\ 
 1 & 0 & 0 & 0 & 0 & 1 & 0 & 0 & 0 & 0 & 0 & 1 & 0 & 0  \\ 
 0 & 1 & 0 & 0 & 0 & 1 & 0 & 0 & 0 & 0 & 0 & 0 & 1 & 0  \\ 
 0 & 0 & 1 & 1 & 0 & 0 & 0 & 0 & 0 & 1 & 0 & 0 & 0 & 1  \\ 
 \end{pmatrix}, \quad b \, = \, \begin{pmatrix} 
 3  \\ 
 7  \\ 
 4  \\ 
 5  \\ 
 5  \\ 
 5  \\ 
 3  \\ 
 5  \\ 
 5  \\ 
 5  \\ 
 7  \\ 
 \end{pmatrix}.\]
 We note that $b$ does not lie in the interior of a full dimensional cell of ${\cal C}_A$: the facet hyperplane arrangement of the permutahedron is not simple (see Proposition \ref{prop:simple}). Still, because $Q$ is a simple polytope, the adjoint polynomial $\alpha_C(x)$ can be computed using the formula in \eqref{eq:alphaCVC}. It has degree 11, and all its coefficients are equal: 
 \begin{equation*}
 \small
     \begin{matrix} x_{1}x_{2}x_{3}x_{4}x_{5}x_{6}x_{7}x_{8}x_{9}x_{12}x_{14} + x_{1}x_{2}x_{3}x_{4}x_{5}x_{6}x_{8}x_{9}x_{12}x_{13}x_{14} + x_{1}x_{2}x_{3}x_{4}x_{5}x_{6}x_{8}x_{10}x_{11}x_{12}x_{13} + \\
     x_{1}x_{2}x_{3}x_{4}x_{5}x_{8}x_{9}x_{11}x_{12}x_{13}x_{14} + x_{1}x_{2}x_{3}x_{4}x_{5}x_{8}x_{10}x_{11}x_{12}x_{13}x_{14} + x_{1}x_{2}x_{3}x_{4}x_{6}x_{7}x_{8}x_{9}x_{10}x_{11}x_{14} + \\
     x_{1}x_{2}x_{3}x_{4}x_{6}x_{7}x_{8}x_{9}x_{11}x_{12}x_{14} + x_{1}x_{2}x_{3}x_{5}x_{6}x_{7}x_{9}x_{10}x_{11}x_{13}x_{14} + x_{1}x_{2}x_{3}x_{5}x_{6}x_{7}x_{10}x_{11}x_{12}x_{13}x_{14} + \\
     x_{1}x_{2}x_{3}x_{5}x_{7}x_{8}x_{9}x_{11}x_{12}x_{13}x_{14} + x_{1}x_{2}x_{3}x_{5}x_{7}x_{8}x_{10}x_{11}x_{12}x_{13}x_{14} + x_{1}x_{2}x_{3}x_{6}x_{7}x_{8}x_{9}x_{10}x_{11}x_{13}x_{14} + \\
     x_{1}x_{2}x_{3}x_{6}x_{7}x_{8}x_{9}x_{11}x_{12}x_{13}x_{14} + x_{1}x_{2}x_{4}x_{5}x_{6}x_{7}x_{8}x_{9}x_{10}x_{11}x_{14} + x_{1}x_{3}x_{4}x_{5}x_{6}x_{7}x_{8}x_{10}x_{11}x_{12}x_{13} + \\
     x_{1}x_{3}x_{4}x_{5}x_{6}x_{7}x_{9}x_{10}x_{11}x_{13}x_{14} + x_{1}x_{3}x_{4}x_{5}x_{6}x_{7}x_{10}x_{11}x_{12}x_{13}x_{14} + x_{1}x_{4}x_{5}x_{6}x_{7}x_{8}x_{9}x_{10}x_{11}x_{13}x_{14} + \\
     x_{2}x_{3}x_{4}x_{5}x_{6}x_{7}x_{8}x_{9}x_{10}x_{12}x_{14} + x_{2}x_{3}x_{4}x_{5}x_{6}x_{8}x_{9}x_{10}x_{11}x_{12}x_{13} + x_{2}x_{3}x_{4}x_{5}x_{6}x_{8}x_{9}x_{10}x_{12}x_{13}x_{14} + \\
     x_{2}x_{4}x_{5}x_{6}x_{7}x_{8}x_{9}x_{10}x_{11}x_{12}x_{14} + x_{3}x_{4}x_{5}x_{6}x_{7}x_{8}x_{9}x_{10}x_{11}x_{12}x_{13} + x_{4}x_{5}x_{6}x_{7}x_{8}x_{9}x_{10}x_{11}x_{12}x_{13}x_{14}.
     \end{matrix}
 \end{equation*}
 This is found using \texttt{adjoint\_x(A,b)}, as in Example \ref{ex:sec3}. The command \texttt{get\_santalo\_point} computes the Santal\'o point by first solving the likelihood equations for random parameters:
\begin{minted}{julia}
A, b, W, c = free_representation(Q) 
solve_likelihood_startsystem(A,b)
\end{minted}
The first line computes the representations $P_b = \{x \geq 0, Ax = b\}$ and $Q = \{Wy + c \geq 0\}$. 
The result of line 2 shows that the ML degree of the Wachspress model of the permutahedron is 569.  Interestingly, we find that the Santal\'o point of the permutahedra of dimensions 2, 3, 4 and 5 is $A^\dag \cdot b$. That is, it is the closest point to the origin satisfying $Ax = b$. 
\end{example}

\subsection{Tracking paths on Santal\'o patches}
Suppose the Santal\'o point $x^*(b_0)$ of $P_{b_0}$ was computed for some $b_0 \in {\rm int}(C)$, where $C \in {\cal C}_A$ is a cell. We are interested in computing $x^*(b_1)$ for some $b_1 \in C$ contained in the same cell. Note that $b_1$ is not necessarily contained in the interior of $C$. In particular $P_{b_1}$ is not necessarily simple.  This time, the parametric equations depend only on $b$: 
\begin{equation}
F(x;b) \, = \, \left( \begin{matrix} B_C^T \cdot \left ( \frac{\partial_{x_i} \alpha_C}{\alpha_C} - \frac{1}{x_i} \right )_{i \in {\cal F}_C} \\
Ax - b \end{matrix} \right ) \, = \, 0.
\end{equation}
The path $\gamma$ is $\gamma(t) = (1-t)\cdot b_0 + t \cdot b_1$. At every $t \in [0,1]$ the solution path $(t,x(t))$ described by the Santal\'o point is smooth: it is a regular solution to the equations $F(x; \gamma(t))$ by convexity of the dual volume. In this homotopy, we track only one path, and all computations can be done over the real numbers. This feature of our problem makes the procedure extra efficient. 

\begin{example}
    In our running Example \ref{ex:pentagon1}, we can set $b_0 = (1,4/5,4/5)$, $b_1 = (1,1,4/5)$, see Figure \ref{fig:pentagon1}. The fiber $P_{b_1}$ is a quadrilateral: $b_1$ lies on the boundary of the pentagonal cell in ${\cal C}_A$. As $t$ moves from $0$ to $1$, the Santal\'o point $x^*(\gamma(t))$ of $P_{\gamma(t)}$ describes a path on the Santal\'o patchwork from Figure \ref{fig:SPpentagon}. In the $(y_1,y_2)$-plane, this is a path in the interior of pentagon $Q_{\gamma(t)}$ which degenerates continuously to a quadrilateral. The Santal\'o point $x^*(b_0)$ was computed in Example \ref{ex:6-1}. The command \texttt{santalo\_path} computes $x^*(b_1)$ from $x^*(b_0)$:
\begin{minted}{julia}
A = [1 1 1 1 1; 2 1 0 1 0; 1 2 0 0 1]
b0 = 1//5*[5; 4; 4]; b1 = 1//5*[5; 5; 4];
x0 = get_santalo_point(A,b0)
x1 = santalo_path(A,b0,b1,x0)    
\end{minted}
The result is $x^*(b_1) \approx  ( 0.291,0.181,0.145,0.237,0.146)$. 
\end{example}

\subsection{Solving a linear program} \label{sec:63}

The logarithm of the dual volume function $\log V(x)$ appears as the \emph{universal barrier function} in the convex optimization literature \cite[Section 2.5]{nesterov1994interior}. The name is motivated by the fact that the dual volume is a barrier function for \emph{any} convex body, and it has the favorable property of being \emph{self-concordant} \cite[Theorem 2.5.1]{nesterov1994interior}. A disadvantage is that the dual volume function can be complicated in general, which limits the use of universal barrier functions for practical purposes. However, for linear programs, whose feasible regions are convex polytopes, the function $V(x)$ is rational, and can be computed using Proposition \ref{prop:adjformula}. Here we illustrate by means of an example how to use $V(x)$ in an interior point algorithm.

\begin{example}
    We solve a linear program using the data $A, b$ from Example \ref{ex:AB}: 
    \[ \text{ Minimize } c^Tx \quad \text{ subject to} \quad A\, x = b, \quad x \geq 0. \]
    Here $c = (-1,7,7,-9,2)^T$ is a cost vector we selected at random. The minimum is attained at the vertex $x_{\rm opt} = (0,0.3,0,0.5,0.2)$ of $P_b \subset \mathbb{R}^5$. We regularize our program using the universal barrier function $\log V(x)$. That is, we introduce a positive regularization parameter $\varepsilon$ and track the optimizer $x_{\rm opt}(\varepsilon)$ of the following program as $\varepsilon$ goes from $0$ to $1$: 
    \[ \text{ Minimize } \varepsilon \cdot c^Tx + (1-\varepsilon)\cdot \log V(x) \quad \text{ subject to} \quad A\, x = b, \quad x \geq 0. \]
    At $\varepsilon = 0$, the minimizer is the Santal\'o point $x_{\rm opt}(0) = x^*(b)$ computed in Example \ref{ex:6-1}. A homotopy with parameter $\varepsilon$ tracks $x_{\rm opt}(\varepsilon)$ for $\varepsilon \in [0,1)$. The equations are 
    \[ F(x;\varepsilon) \, = \, \left( \begin{matrix} B_C^T \cdot \left ( \varepsilon \, c_i + (1-\varepsilon) \cdot \left [ \frac{u_0 \partial_{x_i} \alpha_C}{\alpha_C} - \frac{u_i}{x_i} \right ]   \right )_{i \in {\cal F}_C} \\
Ax - b \end{matrix} \right ) \, = \, 0. \]
    They are obtained by adapting \eqref{eq:homotopy} to the regularized cost function. The following table records the $\log_{10}$ of the absolute error, rounded to two significant decimal~digits: 
    
    \vspace{0.4cm}
    \begin{tabular}{c|ccccccccc}
    
         $1-\varepsilon$ &    $10^{-1.5}$ &  $10^{-2}$ &  $10^{-2.5}$ &  $10^{-3}$ &  $10^{-3.5}$ &  $10^{-4}$ &  $10^{-4.5}$ &  $10^{-5}$ \\
         \hline
         $\log \lVert x_{\rm opt} - x_{\rm opt}(\varepsilon)\rVert$ & 
$-1.7$ & 
$-2.2$ & 
$-2.7$ & 
$-3.2$ & 
$-3.7$ & 
$-4.2$ & 
$-4.7$ & 
$-5.2$
    \end{tabular}
\end{example}

\section{Future directions} \label{sec:7}

We conclude with a summary of ideas for future research. Two challenges are provided by Conjectures \ref{conj:irred} and \ref{conj:genericadjoint}. More generally, it is interesting to find formulas for the maximum likelihood degree of Wachspress models in terms of the combinatorics of the polytope. 

In the context of linear programming, it is relevant to study the strictly convex objective function $\varepsilon \, c^T x + (1-\varepsilon) \cdot \log V(x)$, for varying values of $\varepsilon \in [0,1]$. It was illustrated in Section \ref{sec:63} that for $\varepsilon \rightarrow 0$ we recover the dual volume objective. For $\varepsilon \rightarrow 1$, we are solving a linear program. One can define a Santal\'o patchwork for each fixed value of $\varepsilon$. We propose to study the degeneration of this $\varepsilon$-dependent Santal\'o patchwork as $\varepsilon$ moves from $0$ to $1$.

Next to their important role in convex optimization, we believe that \emph{generalized Santal\'o points}, meaning critical points of $\log V_{C,u}(x)$ from \eqref{eq:VCu}, can be used for the numerical evaluation of Euler integrals via the saddle point method \cite[Section 5, problem 1]{fourlectures}. 

Another next step is to go \emph{beyond convex polytopes}. The Santal\'o point is well-defined for any full-dimensional convex body. One could start with spectrahedra, which is natural in the context of semidefinite programming. The Santal\'o patchwork of a spectrahedron replaces the Gibbs manifold for entropic regularization \cite{pavlov2023gibbs} when the volumetric barrier is used. 

Finally, we propose to study the broader context of Remark \ref{rem:optvsmle}: which strictly convex functions give rise to interesting semi-algebraic subsets of $\mathbb{R}^n_{>0}$? Furthermore, when and how are these semi-algebraic sets naturally connected to maximum likelihood estimation? 

\section*{Acknowledgements}
We thank Frank Sottile and Bernd Sturmfels for useful conversations. We are grateful to two anonymous referees for their detailed reading and suggestions. 
\bibliographystyle{abbrv}
\small
\bibliography{references.bib}

\begin{thebibliography}{10}

\bibitem{alexandr2023maximum}
Y.~Alexandr and S.~Ho{\c{s}}ten.
\newblock Maximum information divergence from linear and toric models.
\newblock {\em arXiv:2308.15598}, 2023.

\bibitem{arkani2018scattering}
N.~Arkani-Hamed, Y.~Bai, S.~He, and G.~Yan.
\newblock Scattering forms and the positive geometry of kinematics, color and
  the worldsheet.
\newblock {\em Journal of High Energy Physics}, 2018(5):1--78, 2018.

\bibitem{arkani2017positive}
N.~Arkani-Hamed, Y.~Bai, and T.~Lam.
\newblock Positive geometries and canonical forms.
\newblock {\em Journal of High Energy Physics}, 2017(11):1--124, 2017.

\bibitem{bertsimas1997introduction}
D.~Bertsimas and J.~N. Tsitsiklis.
\newblock {\em Introduction to Linear Optimization}.
\newblock Athena Scientific, 1997.

\bibitem{billera1993duality}
L.~J. Billera, I.~M. Gelfand, and B.~Sturmfels.
\newblock Duality and minors of secondary polyhedra.
\newblock {\em Journal of Combinatorial Theory, Series B}, 57(2):258--268,
  1993.

\bibitem{borinsky2023bayesian}
M.~Borinsky, A.-L. Sattelberger, B.~Sturmfels, and S.~Telen.
\newblock Bayesian integrals on toric varieties.
\newblock {\em SIAM Journal on Applied Algebra and Geometry}, 7(1):77--103,
  2023.

\bibitem{breiding2018homotopycontinuation}
P.~Breiding and S.~Timme.
\newblock Homotopycontinuation.jl: A package for homotopy continuation in
  {J}ulia.
\newblock In {\em Mathematical Software--ICMS 2018: 6th International
  Conference, South Bend, IN, USA, July 24-27, 2018, Proceedings 6}, pages
  458--465. Springer, 2018.

\bibitem{catanese2006maximum}
F.~Catanese, S.~Ho{\c{s}}ten, A.~Khetan, and B.~Sturmfels.
\newblock The maximum likelihood degree.
\newblock {\em American Journal of Mathematics}, 128(3):671--697, 2006.

\bibitem{de2010triangulations}
J.~A. De~Loera, J.~Rambau, and F.~Santos.
\newblock {\em Triangulations}.
\newblock Springer Berlin, Heidelberg, 2010.

\bibitem{de2012central}
J.~A. De~Loera, B.~Sturmfels, and C.~Vinzant.
\newblock The central curve in linear programming.
\newblock {\em Foundations of Computational Mathematics}, 12:509--540, 2012.

\bibitem{drton2008lectures}
M.~Drton, B.~Sturmfels, and S.~Sullivant.
\newblock {\em Lectures on algebraic statistics}, volume~39.
\newblock Springer Science \& Business Media, 2008.

\bibitem{duff2019solving}
T.~Duff, C.~Hill, A.~Jensen, K.~Lee, A.~Leykin, and J.~Sommars.
\newblock Solving polynomial systems via homotopy continuation and monodromy.
\newblock {\em IMA Journal of Numerical Analysis}, 39(3):1421--1446, 2019.

\bibitem{gaetz2020positive}
C.~Gaetz.
\newblock Positive geometries learning seminar, canonical forms of polytopes
  from adjoints.
\newblock {\em Unpublished lecture notes, available at
  \url{https://sites.google.com/view/crgaetz/research}}, 2020.

\bibitem{garcia2010linear}
L.~D. Garcia-Puente and F.~Sottile.
\newblock Linear precision for parametric patches.
\newblock {\em Advances in Computational Mathematics}, 33:191--214, 2010.

\bibitem{M2}
D.~R. Grayson and M.~E. Stillman.
\newblock Macaulay2, a software system for research in algebraic geometry.
\newblock Available at \url{http://www2.macaulay2.com}.

\bibitem{guler1996barrier}
O.~G{\"u}ler.
\newblock Barrier functions in interior point methods.
\newblock {\em Mathematics of Operations Research}, 21(4):860--885, 1996.

\bibitem{hartshorne2013algebraic}
R.~Hartshorne.
\newblock {\em Algebraic geometry}, volume~52.
\newblock Springer Science \& Business Media, 2013.

\bibitem{hosten2005solving}
S.~Ho{\c{s}}ten, A.~Khetan, and B.~Sturmfels.
\newblock Solving the likelihood equations.
\newblock {\em Foundations of Computational Mathematics}, 5:389--407, 2005.

\bibitem{huh2013maximum}
J.~Huh.
\newblock The maximum likelihood degree of a very affine variety.
\newblock {\em Compositio Mathematica}, 149(8):1245--1266, 2013.

\bibitem{huh2014likelihood}
J.~Huh and B.~Sturmfels.
\newblock Likelihood geometry.
\newblock {\em Combinatorial algebraic geometry}, 2108:63--117, 2014.

\bibitem{irving2014geometry}
C.~Irving and H.~Schenck.
\newblock {Geometry of Wachspress surfaces}.
\newblock {\em Algebra \& Number Theory}, 8(2):369--396, 2014.

\bibitem{kohn2021adjoints}
K.~Kohn, R.~Piene, K.~Ranestad, F.~Rydell, B.~Shapiro, R.~Sinn, M.-S. Sorea,
  and S.~Telen.
\newblock Adjoints and canonical forms of polypols.
\newblock {\em arXiv:2108.11747}, 2021.

\bibitem{kohn2020projective}
K.~Kohn and K.~Ranestad.
\newblock {Projective geometry of Wachspress coordinates}.
\newblock {\em Foundations of Computational Mathematics}, 20:1135--1173, 2020.

\bibitem{kohn2020moment}
K.~Kohn, B.~Shapiro, and B.~Sturmfels.
\newblock Moment varieties of measures on polytopes.
\newblock {\em Annali della Scuola Normale Superiore di Pisa (Classe Scienze),
  Serie V}, 21:739--770, 2020.

\bibitem{lam2022invitation}
T.~Lam.
\newblock An invitation to positive geometries.
\newblock {\em arXiv:2208.05407}, 2022.

\bibitem{fourlectures}
S.-J. Matsubara-Heo, S.~Mizera, and S.~Telen.
\newblock {Four lectures on {E}uler integrals}.
\newblock {\em SciPost Phys. Lect. Notes}, page~75, 2023.

\bibitem{meyer1998santalo}
M.~Meyer and E.~Werner.
\newblock The {S}antal{\'o}-regions of a convex body.
\newblock {\em Transactions of the American Mathematical Society},
  350(11):4569--4591, 1998.

\bibitem{moraga2021bounding}
J.~Moraga and H.~S{\"u}{\ss}.
\newblock Bounding toric singularities with normalized volume.
\newblock {\em arXiv:2111.01738}, 2021.

\bibitem{nesterov1994interior}
Y.~Nesterov and A.~Nemirovskii.
\newblock {\em Interior-point polynomial algorithms in convex programming}.
\newblock SIAM, 1994.

\bibitem{OSCAR}
Oscar -- open source computer algebra research system, version 0.14.0, 2024.

\bibitem{pavlov2023gibbs}
D.~Pavlov, B.~Sturmfels, and S.~Telen.
\newblock Gibbs manifolds.
\newblock {\em Information Geometry}, pages 1--27, 2023.

\bibitem{mathrepo}
D.~Pavlov and S.~Telen.
\newblock Math{R}epo {S}antal\'o.
\newblock \url{https://mathrepo.mis.mpg.de/Santalo}, 2023.

\bibitem{shafarevich}
I.~R. Shafarevich.
\newblock {\em {Basic algebraic geometry}}.
\newblock Springer, Berlin, 2013.

\bibitem{sommese2005numerical}
A.~J. Sommese, C.~W. Wampler, et~al.
\newblock {\em The Numerical solution of systems of polynomials arising in
  engineering and science}.
\newblock World Scientific, 2005.

\bibitem{sturmfels2024toric}
B.~Sturmfels, S.~Telen, F.-X. Vialard, and M.~von Renesse.
\newblock Toric geometry of entropic regularization.
\newblock {\em Journal of Symbolic Computation}, 120:102221, 2024.

\bibitem{sullivant2018algebraic}
S.~Sullivant.
\newblock {\em Algebraic statistics}, volume 194.
\newblock American Mathematical Soc., 2018.

\bibitem{warren1996barycentric}
J.~Warren.
\newblock Barycentric coordinates for convex polytopes.
\newblock {\em Advances in Computational Mathematics}, 6:97--108, 1996.

\end{thebibliography}

\noindent{\bf Authors' addresses:}
\medskip

\noindent Dmitrii Pavlov, MPI-MiS Leipzig
\hfill {\tt dmitrii.pavlov@mis.mpg.de}

\noindent Simon Telen, MPI-MiS Leipzig
\hfill {\tt simon.telen@mis.mpg.de}
\end{document}